\documentclass[11pt]{article}

\usepackage{amsmath,amsthm,thmtools,amssymb,amsfonts, graphicx}
\usepackage{graphics}
\usepackage{authblk}
\usepackage{upgreek}
\usepackage{amsrefs}
\usepackage[hypertexnames=false]{hyperref}
\usepackage{cleveref}
\usepackage[left=0.5in,right=1in,top=0.7in,bottom=1in]{geometry}
\usepackage{longtable}
\usepackage{epstopdf,tcolorbox}
\usepackage{empheq}
\usepackage{authblk}
\theoremstyle{plain}
\newtheorem{theorem}{Theorem}[section]

\newtheorem{corollary}[theorem]{Corollary}
\newtheorem{definition}[theorem]{Definition}

\newtheorem{lemma}[theorem]{Lemma}
\newtheorem{proposition}[theorem]{Proposition}

\theoremstyle{definition}
\newtheorem{remark}[theorem]{Remark}
\numberwithin{equation}{section}
\newcommand{\diff}{\mathop{}\!\mathrm{d}}

\DeclareMathOperator{\tr}{tr}
\DeclareMathOperator{\divr}{div}
\DeclareMathOperator{\spn}{span}

\DeclareMathOperator{\sgn}{sgn}

\newcommand{\R}{\mathbb{R}}
\renewcommand{\S}{\mathbb{S}}

\newcommand{\p}{\partial}

\renewcommand{\th}{\theta}

\newcommand{\om}{\omega}

\newcommand{\rb}{\bar{r}}
\newcommand{\zb}{\bar{z}}
\newcommand{\xb}{\bar{x}}

\newcommand{\bke}[1]{\left ( #1 \right )}
\newcommand{\bkt}[1]{\left [ #1 \right ]}

\newcommand{\EQ}[1]{\begin{equation} #1 \end{equation}}

\newcommand{\EQS}[1]{\begin{equation}\begin{split} #1 \end{split}\end{equation}}

\newcommand{\cF}{\mathcal{F}}

\usepackage{accents}

\newcommand{\pd}{\partial}

\newcommand\De{\Delta}
\newcommand\Ga{\Gamma}
\newcommand\Om{\Omega}
\renewcommand{\div}{\mathop{\rm div}}

\title{On the regularity of axisymmetric, swirl-free solutions of the Euler equation in four and higher dimensions}

\author[1]{Evan Miller}
\author[2]{Tai-Peng Tsai}
\affil[1]{University of Maine,
Department of Mathematics and Statistics}
\affil[1]{evan.miller1@maine.edu}
\affil[2]{University of British Columbia, Department of Mathematics}
\affil[2]{ttsai@math.ubc.ca}

\begin{document}

\maketitle

\begin{abstract}
In this paper, we consider axisymmetric, swirl-free solutions of the Euler equation in four and higher dimensions. We show that in dimension $d\geq 4$, axisymmetric, swirl-free solutions of the Euler equation have properties which could allow finite-time singularity formation of a form that is excluded when $d=3$, and we prove a conditional blowup result for axisymmetric, swirl-free solutions of the Euler equation in dimension $d\geq 4$. The condition which must be imposed on the solution in order to imply blowup becomes weaker as $d\to +\infty$, suggesting the dynamics are becoming much more singular as the dimension increases.
\end{abstract}

\setcounter{tocdepth}{2}
\tableofcontents

\section{Introduction}
The incompressible Euler equation is one of the fundamental equations of fluid dynamics, and is given by
\begin{align}
    \partial_t u+(u\cdot\nabla)u
    +\nabla p &=0 \\
    \nabla\cdot u&=0,
\end{align}
where $u\in\mathbb{R}^d$ is the fluid velocity, and $p$ is the pressure, which can be determined from the velocity using the divergence free constraint, yielding
\begin{equation}
    -\Delta p=\sum_{i,j=1}^d
    \frac{\partial u_j}{\partial x_i}
    \frac{\partial u_i}{\partial x_j}.
\end{equation}
This means that the Euler equation can be expressed in terms of the Helmholtz projection as
\begin{equation}
    \partial_t u
    +P_{df}\left((u\cdot\nabla)u\right)
    =0.
\end{equation}

It is a classical result that for sufficiently smooth initial data, the Euler equation has global smooth solutions in two dimensions. In three dimensions, it is a major open question in nonlinear PDE whether smooth solutions of the Euler equation can blowup in finite-time. One of the main partial regularity results is the Beale-Kato-Majda criterion, which states that the $L^1_t L^\infty_x$ norm of the vorticity controls regularity \cite{BKM}, and so if a smooth solution of the Euler equation blows up in finite-time $T_{max}<+\infty$, then
\begin{equation}
    \int_0^{T_{max}}
    \|\vec{\omega}(\cdot,t)\|_{L^\infty}
    \diff t=+\infty.
\end{equation}

A special case of the Euler equation in three dimensions is that of axisymmetric, swirl-free solutions. Axisymmetric, swirl free vector fields have the form
\begin{equation}
    u(x)=u_r(r,z)e_r+u_z(r,z)e_z,
\end{equation}
where
\begin{align}
    r&=\sqrt{x_1^2+x_2^2} \\
    z&=x_3 \\
    e_r&=\frac{(x_1,x_2,0)}{r} \\
    e_z&= e_3,
\end{align}
and this class of vector fields is preserved by the dynamics of the Euler equation.
The scalar vorticity for axisymmetric, swirl-free solutions is $\omega=\partial_ru_z-\partial_zu_r$, and satisfies the evolution equation
\begin{equation}
    \partial_t\omega+(u\cdot\nabla)\omega
    -\frac{u_r}{r}\omega=0.
\end{equation}
This implies that the quantity $\frac{\omega}{r}$ is transported by the flow with
\begin{equation}
    (\partial_t+u\cdot\nabla)\frac{\omega}{r}=0,
\end{equation}
and this is enough to guarantee global regularity subject to reasonable hypotheses on the initial data.
Ladyzhenskaya proved the first global regularity result for axisymmetric, swirl-free fluids in three dimensions, proving global regularity for axisymmetric, swirl-free solutions of the Navier--Stokes equation \cite{Ladyzhenskaya}.
Ukhovskii and Yudovich proved global regularity \cite{Yudovich} for axisymmetric, swirl-free solutions of the Euler equation in three dimensions with initial vorticity satisfying $\omega^0, \frac{\omega^0}{r}\in 
L^2\cap L^\infty$, which was later generalized by Serfati \cite{Serfati} and Saint-Raymond \cite{SaintRaymond}.
Danchin relaxed this requirement \cites{Danchin1,Danchin2} to
initial data satisfying
$\omega^0\in L^{3,1}\cap L^\infty,
\frac{\omega^0}{r}\in L^{3,1}$.
We will note that 
in the standard regularity class for well posedness of strong solutions to the three dimensional Euler equation, $H^s\left(\mathbb{R}^3\right), s>\frac{5}{2}$,
the conditions by Danchin hold automatically,
and so these conditions do not involve any assumptions beyond sufficient regularity of the initial data.

In this paper, we will consider a generalization of axisymmetric, swirl-free solutions of the Euler equation to four and higher dimensions. For $d\geq 4$, we will consider solutions of the form
\begin{equation}
    u(x,t)=u_r(r,z,t)e_r+u_z(r,z,t)e_z,
\end{equation}
where
\begin{align}
    r&=\sqrt{x_1^2+...+x_{d-1}^2} \\
    z&=x_d \\
    e_r&=\frac{(x_1,...,x_{d-1},0)}{r} \\
    e_z&= e_d,
\end{align}
The scalar vorticity is again given by $\omega=\partial_ru_z-\partial_zu_r$, and now satisfies the evolution equation
\begin{equation}
    \partial_t\omega+(u\cdot\nabla)\omega
    -k\frac{u_r}{r}\omega=0,
\end{equation}
where $k=d-2$.
This implies that the quantity $\frac{\omega}{r^k}$ is transported by the flow with
\begin{equation}
    (\partial_t+u\cdot\nabla)\frac{\omega}{r^k}=0.
\end{equation}

\begin{remark}

This immediately opens the possibility of finite-time singularity formation when $d\geq 4$, because the transported quantity $\frac{\omega}{r^k}$ is not necessarily bounded when $k\geq 2$.
In general, for sufficiently smooth, axisymmetric, swirl-free velocities, the vorticity must vanish linearly at the axis.
To be more precise, we will show that if $u\in H^s\left(\mathbb{R}^d\right), s>2+\frac{d}{2}$, then $\frac{\omega}{r}\in L^\infty$.
For $d\geq 4$, by contrast, $\frac{\omega}{r^k}$ may be unbounded for even a Schwartz class vector field.

The fact that there is some control on the advected quantity $\frac{\omega}{r}$ of the kind imposed by Ukhovskii and Yudovich \cite{Yudovich} or Danchin \cites{Danchin1,Danchin2} is essential to the guaranteeing global regularity.
In fact, for solutions in which the vorticity is only $C^\epsilon,$ for a small $\epsilon>0$, Elgindi recently proved finite-time blowup for solutions of the axisymmetric, swirl-free Euler equation in three dimensions \cite{Elgindi}. 
Note that for these solutions, $\frac{\omega}{r}$ is singular at the $z$-axis, $r=0$, even before the blowup time.
When $d\geq 4$, the advected quantity $\frac{\omega}{r^k}$ may be singular at the $z$-axis even for smooth solutions,
so this gives us a good reason to be optimistic that there will be examples of finite-time blowup for smooth, axisymmetric, swirl-free solutions of the Euler equation in four and higher dimensions.
If the quantity $\frac{\omega}{r^k}$ being advected by the flow does not have any reasonable control, then there is no barrier to blowup from axisymmetry, and we can expect blowup for smooth solutions in $d\geq 4$ analogous to the blowup for strong---but not smooth---solutions proven by Elgindi when $d=3$.
\end{remark}

\medskip

It would seem natural to expect that for $d\geq 4$, in the case where $\frac{\omega^0}{r^k}\in L^\infty$, then there will be global regularity by similar arguments guaranteeing global regularity in $d=3$, when $\frac{\omega^0}{r}\in L^\infty$.
This is the case when $d=4$, and we have the following result.

\begin{theorem} \label{4DregularityIntro}
Suppose the initial data $u^0\in H^s_{df}
\left(\mathbb{R}^4\right), s>4$
is axisymmetric and swirl-free and $\frac{\omega^0}{r^2}\in L^1\cap L^\infty$.
Then there exists a global smooth solution of the Euler equation $u\in C\left([0,+\infty),H^s_{df}
\left(\mathbb{R}^4\right)\right)
\cap C^1\left([0,+\infty),H^{s-1}_{df}
\left(\mathbb{R}^4\right)\right)$. Furthermore, we have a bound on vorticity,
with for all $R>0$ and for all $0\leq t<+\infty$,
\begin{equation} 
    \|\omega(\cdot,t)\|_{L^\infty}
    \leq
    \max\left(\left\|\omega^0
    \right\|_{L^\infty(\mathcal{C}_R^c)},
    R^2 \left\|\frac{\omega^0}{r^2}
    \right\|_{L^\infty(\mathcal{C}_R)}\right)
    \exp(2\mu t),
\end{equation}
where
\begin{equation}
    \mu=\frac{C}{R} \left\|\frac{\omega^0}{r^k}
    \right\|_{L^\infty}^\frac{1}{2}
    \left(\left\|\omega^0
    \right\|_{L^{1}(\mathcal{C}_R^c)}
    +R^2\left\|\frac{\omega^0}{r^2}
    \right\|_{L^{1}(\mathcal{C}_R)}
    \right)^\frac{1}{2},
\end{equation}
with the cylinder 
$\mathcal{C}_R\in \mathbb{R}^d$ given by
\[
    \mathcal{C}_R=\left\{x\in \mathbb{R}^d:
    |x'|<R\right\}.
\]
\end{theorem}

\begin{remark}
    This paper replaces an earlier version by the first author that did not include this result. In the meantime, Theorem \ref{4DregularityIntro} was proven independently and first published by Choi, Jeong, and Lim \cite{ChoiJeongLim}. We include it in this paper nonetheless, because it fits naturally into the analysis of upper and lower bounds on the singular behaviour of axisymmetric solutions the Euler equation in higher dimensions. 
    
    Lim and Jeong further refined the analysis in \cite{LimJeongARMA}, proving global regularity for axisymmetric swirl free solutions of the Euler equation whenever $d\leq 6$ and $\frac{\omega^0}{r^{d-2}}\in L^\infty$. For $3\leq d\leq 5$, they proved an upper bound on the growth rate of the form 
    \begin{equation}
    \|\omega(\cdot,t)\|_{L^\infty}
    \leq 
    C(1+t)^\frac{4}{6-d},
    \end{equation}
    and for $d=6$, the bound
    \begin{equation}
    \|\omega(\cdot,t)\|_{L^\infty}
    \leq 
    C e^{\kappa t},
    \end{equation}
in all cases under the assumption that $\frac{\omega^0}{r^{d-2}}\in L^\infty$. These results supersede Theorem \ref{5DregularityIntro} below for $d=5,6$, but the regularity criterion in Theorem \ref{5DregularityIntro} remains relevant when $d\geq 7$, in which case the finite-time blowup problem is open.
\end{remark}

\begin{theorem} \label{5DregularityIntro}
Suppose $u\in C\left([0,T_{max}),H^s_{df}
\left(\mathbb{R}^d\right)\right)
\cap C^1\left([0,T_{max}),H^{s-1}_{df}
\left(\mathbb{R}^d\right)\right),
d\geq 5, s>2+\frac{d}{2}$ is an axisymmetric, swirl-free solution of the Euler equation,
with finite-time blowup at $T_{max}<+\infty$,
and that $\frac{\omega^0}{r^k}\in L^1\cap L^\infty$.
Then for all $0\leq t<T_{max}$,
\begin{equation}
    \|\omega(\cdot,t)
    \|_{L^{1}\left(\mathbb{R}^d\right)}
    \geq 
    \left(\frac{M_d}
    {\left\|\frac{\omega^0}{r^k}
    \right\|_{L^\infty}^\frac{d-2}{d-4}
    \left\|\frac{\omega^0}{r^k}
    \right\|_{L^1}^\frac{2}{d-4}}\right)
    \frac{1}
    {\left(T_{max}-t\right)^{2\frac{d-2}{d-4}}},
\end{equation}
where $M_d>0$ is a constant depending only on the dimension.
Furthermore,
\begin{equation}
    \int_0^{T_{max}}\left\|u_r^+(\cdot,t)
    \right\|_{L^\infty} \diff t
    = +\infty,
\end{equation}
where $u_r^+=\max(u_r,0)$.
\end{theorem}

We will note that in terms of scaling these regularity criteria are substantially stronger than the Beale-Kato-Majda criterion, and so any blowup of this kind must happen in the bulk of the flow, at least to a larger degree than Beale-Kato-Majda on its own would require.
These regularity criteria follow from the fact that when $\frac{\omega^0}{r^k}\in L^\infty$, then there can only be finite-time blowup at some time $T_{max}<+\infty$ if a fluid trajectory runs off to spatial infinity with $r(t)\to\infty$ as $t\to T_{max}$.

As further evidence of potential singularity formation in very high dimensions, we prove that any solution of the Euler equation in $\mathbb{R}^d$ with a particular geometry---that is, axisymmetric, swirl-free, with a scalar vorticity $\omega^0(r,z)$ that is odd in $z$ and satisfying $\omega^0(r,z)\geq 0$ for all $r,z>0$--- blows up in finite time as long as a large enough proportion of energy remains in the vertical velocity component. Furthermore, the proportion of energy remaining in the vertical component necessary to guarantee finite-time blowup shrinks at a rate $\frac{1}{d}$, making the requirement very weak in high dimensions.

\begin{definition}
    For an axisymmetric, swirl-free solution of the incompressible Euler equation $u\in C\left([0,T_{max});
H^s_* \left(\mathbb{R}^d\right)\right)
\cap C^1\left([0,T_{max}),H^{s-1}_*
\left(\mathbb{R}^d\right)\right), 
d\geq 3, s>2+\frac{d}{2}$, 
for all $0\leq t<T_{max}$, let
\begin{align}
K_r(t)&=\int_0^\infty\int_0^\infty 
r^{d-2}u_r(r,z,t)^2 \diff r \diff z \\
K_z(t)&= \int_0^\infty\int_0^\infty 
r^{d-2}u_z(r,z,t)^2 \diff r \diff z,
\end{align}
and let the total energy of the system be given by
\begin{equation}
    K_0=K_r(0)+K_z(0).
\end{equation}
Note that for all $0\leq t<T_{max}$,
\begin{equation}
    K_r(t)+K_z(t)
    =
    \frac{1}{m_{d-2}}
    \|u(\cdot,t)\|_{L^2}^2
    =
    \frac{1}{m_{d-2}}
    \left\|u^0\right\|_{L^2}^2
    =
    K_0,
\end{equation}
due to conservation of energy, so the radial and vertical components of energy always add to the initial total energy, although energy can be exchanged between these two components.
\end{definition}

\begin{theorem} \label{ConditionalBlowupThmIntro}
    Suppose $u\in C\left([0,T_{max});
H^s_* \left(\mathbb{R}^d\right)\right)
\cap C^1\left([0,T_{max}),H^{s-1}_*
\left(\mathbb{R}^d\right)\right), 
d\geq 4, s>2+\frac{d}{2}$, is a solution of the Euler equation and that $\omega^0(r,z)$ is odd in $z$, with for all $r,z>0$
\begin{equation}
    \omega^0(r,z)\geq 0,
\end{equation}
and that $\omega^0$ is not identically zero.
Further suppose that there exists $\epsilon>0$ such that for all $0\leq t<T_{max}$,
\begin{equation}
    K_z(t)\geq \frac{1+\epsilon}{d}K_0.
\end{equation}
Then this solution of the Euler equation blows up in finite-time with
\begin{equation}
    T_{max}\leq \frac{1}{\epsilon K_0} \int_0^\infty \int_0^\infty
    r^{d-1}z \omega^0(r,z) \diff r\diff z.
\end{equation}
\end{theorem}

\begin{remark}
    Note that when the dimension is very large, the amount of energy that needs to remain in the vertical component in order to guarantee blowup becomes extremely small, which suggests increasingly singular behaviour in very high dimensions. This is consistent with the work of Drivas and Elgindi \cite{DrivasElgindi}, which established, in a very different geometry, finite-time blowup in the infinite-dimensional formal limit $d\to +\infty$ for the Euler equation on the torus $\mathbb{T}^d$. The relationship between the finite-time blowup problem for the axisymmetric Euler equation and the dimension was also studied numerically by Hou and Zhang \cite{HouZhang} and, in the viscous case, by Hou \cite{Hou}.
\end{remark}

The paper will be structured as follows. In section \ref{DefinitionSection}, we will go over some key definitions and as well as the notation used in the paper.
In section \ref{AxisymEqnSection}, we will show that the class of axisymmetric, swirl-free vector fields is preserved in four and higher dimensions by the dynamics of the Euler equation and derive the evolution equation for the vorticity.
In section \ref{BiotSavartSection}, we will give a Biot-Savart law for recovering the velocity from the vorticity.
In section \ref{RegCritSection}, we will discuss the regularity theory in four and higher dimensions when $\frac{\omega^0}{r^k}$ is bounded, establishing regularity criteria and proving Theorems \ref{4DregularityIntro} and \ref{5DregularityIntro}.
In section \ref{AntiParallelTubesSection}, we will consider the case of a vorticity with certain symmetry and sign conditions that corresponds to colliding vortex tubes, and prove some preliminary results that suggest the possibility of growth in very large dimensions, proving Theorem \ref{ConditionalBlowupThmIntro}.
In appendix \ref{StreamFunctionSection}, we will express the Biot-Savart law in terms of a stream function.

\begin{remark}
It must be stated at the outset that the Cauchy problem for the Euler equation in four and higher dimensions is completely unphysical. Real, physical fluid flows are three dimensional. In some cases, the flow of actual physical fluids is close enough to being perfectly two dimensional for the two dimensional equation to be a reasonable model, but it is almost inconceivable that the Euler equation in four or higher dimensions could be a reasonable model for any physical fluids. Even in the case of relativistic fluids, where time is not neatly separable from space, the Lorentzian structure means the problem is still fundamentally $3+1$ dimensional, as it is in the classical case if time is treated as a dimension.

Nonetheless, the question of finite-time blowup of the Euler equation in four and higher dimensions is scientifically interesting beyond purely abstract, mathematical curiosity. The dramatic qualitative differences between two and three dimensional fluid mechanics show that turbulence has a very fundamental dependence on dimension. For this reason, any advance on the finite-time blowup of the Euler (or Navier--Stokes) equation in four and higher dimensions---which seems much more within reach than in the three dimensional case---could shed significant light on the possibility of blowup in three dimensions, by allowing the study of the dependence of mechanisms for finite-time blowup on the dimension.
\end{remark}

\section{Definitions and notation}
\label{DefinitionSection}

Throughout this paper: 
\[
  3 \leq d = \mbox{ the spatial dimension}, \quad  k=d-2.
\]

For axisymmetric solutions, it is useful to have notation for the decomposition
$\mathbb{R}^d=\mathbb{R}^{d-1}\times \mathbb{R}$:
\[
    \R^d \ni x=(x',x_d), \qquad x'=(x_1,...,x_{d-1}) \in \R^{d-1}.
\]
We can define the generalized cylindrical coordinates $(r,z)$ for this decomposition by
\[
    r = |x'|, \qquad 
    z = x_d,
\]
with associated unit vectors
\[
    e_r =\left(\frac{x'}{|x'|},0\right), \qquad
    e_z = e_d
\]
where $\left\{e_1,...,e_d\right\}$ denote the standard basis for $\mathbb{R}^d$.
We will say that a vector field $u:\mathbb{R}^d \to \mathbb{R}^d$ is axisymmetric and swirl-free if
\[
    u(x)=u_r(r,z)e_r+u_z(r,z)e_z,
\]
and we will say that a scalar function
$f:\mathbb{R}^d\to \mathbb{R}$ is axisymmetric if
\[
    f(x)=\Tilde{f}(r,z).
\]
We will sometimes slightly abuse notation by equating $f$ with $\Tilde{f}$.
For derivatives in terms of the spatial variables, we will often use the shorthand
\[
    \partial_i=
    \frac{\partial}{\partial x_i},
\]
and likewise for $\partial_r$ and $\partial_z$.


The Lebesgue spaces $L^p\left(\mathbb{R}^d\right)$ will be given the standard definition and norm.
For axisymmetric functions, we will take
$L^p = L^p \left(m_{d-2} r^{d-2} 
\diff r \diff z\right)$,
where $m_d$ is the surface area of the $d$-sphere embedded in $\mathbb{R}^{d+1}$.
So for $1 \leq p < \infty$,
\[
     \|f\|_{L^p} = \|\Tilde{f}\|_{L^p}
    =\left(m_{d-2}\int_0^\infty
    \int_{-\infty}^\infty 
    |\Tilde{f}(r,z)|^p r^{d-2}
    \diff z\diff r
    \right)^\frac{1}{p}.
\]
For $1\leq p,q<+\infty$, we will take the Lorentz spaces $L^{p,q}\left(\mathbb{R}^d\right)$
to be the spaces with the quasinorm
\[
    \|f\|_{L^{p,q}}
    =
    \left(p\int_0^\infty \alpha^{q-1}
    \left|\left\{x\in\mathbb{R}^d:
    |f(x)|>\alpha\right\}
    \right|^\frac{q}{p} \diff\alpha
    \right)^\frac{1}{q}.
\]
In the case $q=\infty$, we will take the quasinorm to be
\[
    \|f\|_{L^{p,\infty}} =
    \left(\sup_{\alpha>0} \alpha^p
    \left|\left\{x\in\mathbb{R}^d:
    |f(x)|>\alpha\right\}\right|
    \right)^\frac{1}{p}.
\]
Note that $L^{p,p}=L^p$, and that for axisymmetric functions, we will again use the measure 
$m_{d-2}r^{d-2}\diff r \diff z$ when working with cylindrical coordinates, so that
\[
    \|f\|_{L^{p,q}} =
    \|\Tilde{f}\|_{L^{p,q}}.
\]

We will define both homogeneous and inhomogeneous Sobolev spaces, $\dot{H}^s$ and $H^s$: for $s > -\frac{d}{2}$,
\[
  \|f\|_{H^s} =
    \left(\int_{\mathbb{R}^d}
    \left(1+|\xi|^2
    \right)^s |\hat{f}(\xi)|^2
    \diff\xi \right)^\frac{1}{2}, \qquad
    \|f\|_{\dot{H}^s} =
    \left(\int_{\mathbb{R}^d}
    |\xi|^{2s}
    |\hat{f}(\xi)|^2 \diff\xi
    \right)^\frac{1}{2},
\]
where $\hat{f}$ denotes the Fourier transform of $f$.

For all $s\in\mathbb{R}$, let
$H^s_{df}\left(\mathbb{R}^d\right)$ be the space of divergence free vector fields:
\[
    H^s_{df} = H^s_{df}\left( \mathbb{R}^d \right) =
    \left\{v\in H^s \left
    (\mathbb{R}^d;\mathbb{R}^d\right)
    :\nabla\cdot v=0\right\}.
\]
Note that the divergence is only a continuous function when $s>1+\frac{d}{2}$, so in general we will say that $\nabla\cdot v=0$ if
\[
    \xi\cdot\hat{v}(\xi)=0
\]
for almost every $\xi\in\mathbb{R}^d$.
We will also take $H^s_{gr}\left(\mathbb{R}^d\right)$ to be the space of gradients,
\[
    H^s_{gr} = H^s_{gr}\left(\mathbb{R}^d \right) =
    \left\{\nabla f \;
    : \; f\in\dot{H}^1(\R^d) \cap
    \dot{H}^{s+1}(\R^d) \right\}.
\]
\begin{remark}
Note that 
$v \in H^s_{gr}$ if and only if
$\hat{v}(\xi)\in\spn(\xi)$ for almost every $\xi\in\mathbb{R}^d$,
and
$v\in H^s_{df}$ if and only if
$\hat{v}(\xi)\in\spn(\xi)^\perp$ for almost every $\xi\in\mathbb{R}^d$.
This leads to the Helmholtz decomposition
\[
    H^s=H^s_{df}\oplus H^s_{gr}.
\]
\end{remark}

Denote the Sobolev space of axisymmetric, swirl-free vector fields as
\[
        H^s_{as} = H^s_{as}\left(\mathbb{R}^d;
        \mathbb{R}^d\right)
        =
        \left\{ v\in H^s(\R^d;\R^d) \; : \;
        v(x)=v_r(r,z)e_r+v_z(r,z)e_z
        \right\}
\]
and the spaces of axisymmetric, swirl-free, divergence free vector fields and axisymmetric, swirl-free gradients as
\[
        H^s_* =
        H^s_{as} \cap H^s_{df}, \qquad
        H^s_{as\&gr} =
        H^s_{as} \cap H^s_{gr}. 
\]

Let $C^k(\mathbb{R}^d)$ denote the space of $k$ times continuously differentiable functions, with bounded derivatives up to order $k$,
with norm
\[
    \|f\|_{C^k} =
    \sup
    _{\substack{\alpha_1,...,\alpha_d\geq 0\\\alpha_1+...+\alpha_d\leq k}}
    \|\partial_1^{\alpha_1}...
    \partial_d^{\alpha_d}f\|_{L^\infty}
\]
We will take the analogous definition when dealing with $C^k$ functions that also depend on a time variable.

\begin{remark}
Note that for all $s>k+\frac{d}{2}$, we have the continuous embedding
\[
    H^s\left(\mathbb{R}^d\right)
    \hookrightarrow
    C^k\left(\mathbb{R}^d\right), \qquad
    \|f\|_{C^k}\leq C_{s,k,d}
    \|f\|_{H^s}.
\]
\end{remark}

We will refer to the anti-symmetric and symmetric parts of the gradient of a vector field $u$ as $A$ and $S$ with
\[
    A_{ij} =\frac{1}{2}\left(
    \partial_i u_j -\partial_j u_i\right), \qquad 
    S_{ij} =\frac{1}{2}\left(
    \partial_i u_j +\partial_j u_i\right),
\]
and define the anti-symmetric gradient operator and symmetric gradient operator likewise:
\[
    \nabla_{asym}u = A, \qquad 
    \nabla_{sym}u = S.
\]
For $d=3$, $A$ can be represented in terms of the vorticity vector $\vec{\omega}=\nabla\times u$ by
\[
    A= \frac{1}{2} \left(
    \begin{array}{ccc}
        0 & \vec{\omega}_3  & -\vec{\omega}_2 \\
        -\vec{\omega}_3 & 0 & \vec{\omega}_1  \\
        \vec{\omega}_2 & -\vec{\omega}_1 & 0
    \end{array}
    \right).
\]
For axisymmetric, swirl-free vector fields $u$,
\[
    A(x) =\frac{1}{2}\omega(r,z) 
    \left(e_r\otimes e_z - e_z\otimes e_r\right)
\]
where the scalar vorticity is
\[
  \om(r,z) = \p_r u_z - \p_z u_r.
\]
For $d=3$,
\[
    \vec{\omega}(x)=-\omega(r,z) e_\theta.
\]


\section{The axisymmetric Euler equation in higher dimensions}
\label{AxisymEqnSection}

In this section, we will show that the axisymmetric, swirl-free Euler equation in four and higher dimensions can be expressed in terms of the evolution of the scalar vorticity,
just as it can in three dimensions.

\subsection{Preliminaries}

We will begin by going over some of the vector calculus results necessary for the analysis of axisymmetric, swirl-free solutions. The computations are mostly routine exercises, however they are very important to the analysis.

\begin{proposition} \label{UgradProp}
Suppose $u\in H_*^s\left(\mathbb{R}^d\right),
s >1+\frac{d}{2}$, is axisymmetric and swirl-free. Then the gradient of $u$ can be expressed as
\begin{multline}
    \nabla u(x)=
    \partial_ru_r(r,z)e_r\otimes e_r
    +\frac{u_r(r,z)}{r} \Tilde{I}_d
    +\partial_z u_z(r,z)e_z\otimes e_z \\
    +\partial_ru_z(r,z) e_r\otimes e_z
    +\partial_zu_r(r,z) e_z\otimes e_r,
\end{multline}
where 
\begin{equation}
    \Tilde{I}_d
    =
    I_d-e_r\otimes e_r -e_z\otimes e_z.
\end{equation}
\end{proposition}

\begin{proof}
First we know that 
\begin{equation}
    u(x)=u_r(r,z)e_r+u_z(r,z)e_z
\end{equation}
We know that
\begin{equation}
    \nabla r=e_r,
\end{equation}
and therefore we can see that
\begin{multline}
    \nabla u(x)=
    \partial_ru_r(r,z)e_r\otimes e_r
    +\partial_z u_z(r,z)e_z\otimes e_z \\
    +\partial_ru_z(r,z) e_r\otimes e_z
    +\partial_zu_r(r,z) e_z\otimes e_r
    +u_r(r,z)\nabla e_r.
\end{multline}
It remains only to show that
\begin{equation}
    \nabla e_r=
    \frac{1}{r}\Tilde{I}_d
\end{equation}
Recall that $e_r=\frac{x'}{|x'|},$ where 
$x'=(x_1,...,x_{d-1},0)$.
Clearly we have
\begin{equation}
    \nabla x'=I_d-e_z\otimes e_z,
\end{equation}
and
\begin{equation}
    \nabla \frac{1}{|x'|}
    =
    \nabla \frac{1}{r}
    =
    -\frac{1}{r^2} e_r.
\end{equation}
Therefore we can conclude that
\begin{align}
    \nabla e_r
    &=
    \frac{1}{r}\nabla x'+ 
    \nabla\frac{1}{|x'|} \otimes x' \\
    &=\frac{1}{r}\left(I_d-e_z\otimes e_z\right)
    -\frac{1}{r^2} e_r\otimes x' \\
    &=
    \frac{1}{r}\left(I_d
    -e_z\otimes e_z -e_r\otimes e_r\right).
\end{align}
This completes the proof.
\end{proof}

\begin{proposition} \label{DivProp}
Suppose $u\in H^s\left(\mathbb{R}^d\right),
s > 1+\frac{d}{2}$, is axisymmetric and swirl-free. Then the divergence of $u$ can be expressed as
\begin{equation}
    \nabla\cdot u(x)
    =\partial_ru_r(r,z)
    +k\frac{u_r(r,z)}{r}
    +\partial_zu_z(r,z),
\end{equation}
where $k=d-2$.
\end{proposition}

\begin{proof}
Taking the identity from Proposition \ref{UgradProp}, we can see that
\begin{align}
    \nabla\cdot u(x)
    &=
    \tr(\nabla u(x)) \\
    &=
    \partial_ru_r+\partial_zu_z
    +\frac{u_r(r,z)}{r} \tr
    \left(\Tilde{I}_d\right) \\
    &=
    \partial_ru_r(r,z)
    +\partial_zu_z(r,z)
    +(d-2)\frac{u_r(r,z)}{r}.
\end{align}
This completes the proof.
\end{proof}

\begin{proposition} \label{AxisymAdvect}
Suppose $u\in H^s\left(\mathbb{R}^d\right),
s > 1+\frac{d}{2}$, is axisymmetric and swirl-free. Then $(u\cdot\nabla)u$ is axisymmetric and swirl free.
\end{proposition}

\begin{proof}
We know by hypothesis that
\begin{equation}
    u(x)=u_r(r,z)e_r+u_z(r,z)e_z
\end{equation}
and by Proposition \ref{UgradProp} that
\begin{multline}
        \nabla u(x)=
    \partial_ru_r(r,z)e_r\otimes e_r
    +\frac{u_r(r,z)}{r} \Tilde{I}_d
    +\partial_z u_z(r,z)e_z\otimes e_z \\
    +\partial_ru_z(r,z) e_r\otimes e_z
    +\partial_zu_r(r,z) e_z\otimes e_r.
\end{multline}
We can see that $\Tilde{I}_d$ is the matrix for the projection onto the subspace $\spn(e_r,e_z)^\perp$,
and so
\begin{align}
    (u\cdot \nabla)u
    &=
    (\nabla u)^{tr}u \\
    &=
    u_r\partial_r u_r e_r
    +u_z \partial_z u_z e_z
    +u_r \partial_r u_z e_z
    + u_z\partial_z u_r e_r \\
    &=
    (u\cdot\nabla u_r) e_r+
    (u\cdot\nabla u_z) e_z.
\end{align}
Recalling the definition of an axisymmetric vector field, this completes the proof.
\end{proof}

\begin{proposition}
Suppose $u\in H^s\left(\mathbb{R}^d\right),
s > 1+\frac{d}{2}$, is axisymmetric and swirl-free.
Then
\begin{equation}
    A(x)=\frac{1}{2}\omega(r,z)
    (e_r\otimes e_z-e_z\otimes e_r),
\end{equation}
where $\omega(r,z)=
\partial_r u_z-\partial_z u_r.$
\end{proposition}
\begin{proof}
Taking the formula for $\nabla u$ from Proposition \ref{UgradProp},
we can see that
\begin{align}
    A(x)
    &=
    \frac{1}{2}\left(
    \partial_ru_z e_r\otimes e_z
    +\partial_zu_r e_z\otimes e_r
    -\partial_ru_z e_z\otimes e_r
    -\partial_zu_r e_r\otimes e_z
    \right) \\
    &=
    \frac{1}{2}(\partial_ru_z-\partial_zu_r)
    (e_r\otimes e_z-e_z\otimes e_r).
\end{align}
This completes the proof.
\end{proof}

\begin{proposition} \label{ErDiv}
For all $d\geq 3$, the divergence of the radial direction $e_r\in\mathbb{R}^d$ is
\begin{equation}
    \nabla\cdot e_r=\frac{k}{r}.
\end{equation}
\end{proposition}
\begin{proof}
Recall that
\begin{equation}
    e_r=\frac{x'}{|x'|}.
\end{equation}
It then follows that 
\begin{align}
    \nabla\cdot e_r
    &=
    \frac{\nabla \cdot x'}{|x'|}
    +x'\cdot \nabla\frac{1}{|x'|} \\
    &=
    \frac{d-1}{r}-x'\cdot 
    \frac{x'}{|x'|^3} \\
    &=
    \frac{d-2}{r}.
\end{align}
This completes the proof.
\end{proof}

\begin{proposition} \label{VectorLaplace}
Suppose $u\in H^s\left(\mathbb{R}^d;
\mathbb{R}^d\right), s> 2+\frac{d}{2}$
is an axisymmetric, swirl-free vector field. 
Then the vector Laplacian can be expressed as
\begin{equation}
    -\Delta u(x)
    =
    \left(-\partial_r^2-\frac{k}{r}\partial_r
    +\frac{k}{r^2}-\partial_z^2\right)
    u_r(r,z) e_r
    +\left(-\partial_r^2
    -\frac{k}{r}\partial_r
    -\partial_z^2\right)
    u_z(r,z) e_z
\end{equation}
\end{proposition}

\begin{proof}
Taking the divergence of the gradient formula in Proposition \ref{UgradProp}, using the fact that $-\Delta u=-\divr\nabla u$, and applying Proposition \ref{ErDiv}, we find that
\begin{equation*}
    -\Delta u
    =
    -\partial_z^2 u_r e_r -\partial_z^2 u_z e_z
    -\partial_r^2 u_z e_z 
    -\frac{k}{r}\partial_r u_z e_z
    -\divr\left(\partial_ru_r
    e_r\otimes e_r\right)
    -\divr \left(\frac{u_r}{r}\Tilde{I}_d\right).
\end{equation*}
Again applying Proposition \ref{ErDiv},
we can see that
\begin{align}
    -\divr
    \left(\partial_ru_r
    (e_r\otimes e_r)\right)
    &=
    \left(-\partial_r^2 u_r
    -\frac{k}{r}\partial_r u_r
    \right)e_r,
\end{align}
and that
\begin{align}
    -\divr\left(\frac{u_r}{r}\Tilde{I}_d\right)
    &=
    -\left(\partial_r\frac{u_r}{r}\right)
    \Tilde{I_d}e_r
    -\frac{u_r}{r}\divr\left(
    I_d-e_r\otimes e_r-e_z\otimes e_z\right)\\
    &=
    \frac{k}{r^2}u_r e_r.
\end{align}
This completes the proof.
\end{proof}

\begin{proposition}
Suppose $f\in H^s, s> 2+\frac{d}{2}$,
is a scalar axisymmetric function with $f(x)=\Tilde{f}(r,z)$.
Then
\begin{equation}
    \nabla f(x)=\partial_r\Tilde{f}(r,z)e_r
    +\partial_z\Tilde{f}(r,z)e_z,
\end{equation}
and
\begin{equation}
    -\Delta f(x)
    =
    -\partial_r^2 \Tilde{f}(r,z)
    -\frac{k}{r}\partial_r \Tilde{f}(r,z)
    -\partial_z^2 \Tilde{f}(r,z).
\end{equation}
\end{proposition}

\begin{proof}
Observe that
\begin{equation}
    \nabla r=e_r,
\end{equation}
and so applying the chain rule,
\begin{equation}
    \nabla f(x)=\partial_r\Tilde{f}(r,z)e_r
    +\partial_z\Tilde{f}(r,z)e_z.
\end{equation}
Taking the divergence of this equation and applying Proposition \ref{DivProp}, this completes the proof.
\end{proof}

\begin{proposition}
\label{LaplaceCalculation}
Suppose $u\in
H^s_{as}\left(\mathbb{R}^d\right),
s> 2+\frac{d}{2}$. Then the vector Laplacian can be expressed as
\begin{equation}
    -\Delta u
    =
    -2\divr\nabla_{asym}u
    -\nabla \nabla\cdot u.
\end{equation}
\end{proposition}
\begin{proof}
First note that for $s>2+\frac{d}{2},
H^s_{as}\left(\mathbb{R}^d\right)
\hookrightarrow 
C^2\left(\mathbb{R}^d\right)$,
so the both sides of the expression are well defined for all $x\in\mathbb{R}^d$.
Let
\begin{equation}
    u(x)=u_r(r,z)e_r+u_z(r,z)e_z
\end{equation}
Recalling that 
\begin{equation}
    \nabla\cdot u=
    \partial_ru_r
    +\frac{k}{r}u_r
    +\partial_z u_z,
\end{equation}
we can compute that
\begin{equation} \label{GradStep}
    \nabla \nabla \cdot u
    =
    \left(\partial_r^2+\frac{k}{r}\partial_r
    -\frac{k}{r^2}\right) u_re_r
    +\partial_z^2 u_z e_z
    +\partial_z\left(\partial_r u_r+\frac{k}{r}u_r\right) e_z
    +\partial_r\partial_zu_z e_r.
\end{equation}
Likewise recalling that 
\begin{equation}
    2\nabla_{asym}u
    =
    (\partial_r u_z-\partial_z u_r)
    (e_r\otimes e_z- e_z\otimes e_r),
\end{equation}
and applying the product rule
\begin{equation}
    \divr\left(fM\right)
    =
    f\divr(M)+M^{tr}\nabla f,
\end{equation}
we find that
\begin{align}
    2\divr\nabla_{asym}u
    &=
    \frac{k}{r}
    (\partial_r u_z-\partial_z u_r)e_z
    +\partial_r(\partial_r u_z
    -\partial_zu_r)e_z
    -\partial_z (\partial_r u_z
    -\partial_zu_r)e_r \\
    &= \label{DivFreeStep}
    \partial_z^2 u_r e_r
    +\left(\partial_r^2+\frac{k}{r}
    \partial_r\right)u_z e_z
    -\frac{k}{r}\partial_z u_r e_z
    -\partial_r\partial_z u_r e_z
    -\partial_z\partial_r u_z e_r.
\end{align}
Adding together \eqref{GradStep} and \eqref{DivFreeStep}, this completes the proof.
\end{proof}

\begin{remark}
Note that the operators $-\Delta, \divr\nabla_{asym}, \nabla\nabla\cdot, (-\Delta)^{-1}$ all preserve the class of axisymmetric, swirl-free vector fields. For the first three operators, this can be observed directly from the computations above. For $(-\Delta)^{-1}$, the result follows in any space where the kernel of the Laplace operator is trivial from the fact that $-\Delta$ preserves this class.
\end{remark}

\begin{proposition} \label{AxisymHelmholtz}
Axisymmetric, swirl free vector fields in dimension $d\geq 3$ have the following Helmholtz decomposition:
\begin{equation}
    H^s_{as}\left(\mathbb{R}^d\right)
    =
    H^s_*\left(\mathbb{R}^d\right)
    \oplus 
    H^s_{as\&gr}\left(\mathbb{R}^d\right),
\end{equation}
for all $s\in\mathbb{R}$.
In particular, for all $v\in H^s_{as}$,
\begin{equation} \label{HelmholtzID}
    v=-2\divr\nabla_{asym}(-\Delta)^{-1}v
    -\nabla \nabla\cdot(-\Delta)^{-1}v, 
\end{equation}
with $-2\divr\nabla_{asym}(-\Delta)^{-1}v
\in H^s_*\left(\mathbb{R}^d\right)$
and
$-\nabla \nabla\cdot(-\Delta)^{-1}v
\in H^s_{as\&gr}\left(\mathbb{R}^d\right)$.
\end{proposition}

\begin{proof}
We know immediately from the classical Helmholtz decomposition that for all 
$u\in H^s_*,\nabla f\in H^s_{as\&gr}$
\begin{equation}
    \left<u,\nabla f\right>_s=0,
\end{equation}
and so the spaces are orthogonal.
It remains only to show that 
\begin{equation}
    H^s_{as}\left(\mathbb{R}^d\right)
    =
    H^s_*\left(\mathbb{R}^d\right)
    +
    H^s_{as\&gr}\left(\mathbb{R}^d\right)
\end{equation}
We will first consider this result for 
$s> \frac{d}{2}$.
Applying Proposition \ref{LaplaceCalculation} to $(-\Delta)^{-1}v$, we find that
\begin{equation}
    v=-2\divr\nabla_{asym}(-\Delta)^{-1}v
    -\nabla \nabla\cdot(-\Delta)^{-1}v.
\end{equation}
It is obvious that $-\nabla \nabla\cdot(-\Delta)^{-1}v 
\in H^s_{as\&gr}$, because $\nabla\cdot v$ is an axisymmetric, scalar function, so it remains only to show that $-2\divr\nabla_{asym}(-\Delta)^{-1}v
\in H^s_*$.
Letting $A=\nabla_{asym}(-\Delta)^{-1}v$,
we can see that
\begin{align}
    \nabla\cdot\left(
    -2\divr A\right)
    &=
    -2\divr^2 A \\
    &=
    -2\sum_{i,j=1}^d
    \partial_i\partial_j A_{ij} \\
    &=0,
\end{align}
due to the anti-symmetry of $A$.
This completes the proof for $s> \frac{d}{2}$.

We should note that the formula for the Laplacian in Proposition \ref{LaplaceCalculation} holds pointwise only for $C^2$ vector fields, so $(-\Delta)^{-1}v$ must be at least $C^2$ for this identity to hold classically.
This is why we require that $s>\frac{d}{2}$ for the identity to hold pointwise.
It remains true, however, that $-2\divr\nabla_{asym}(-\Delta)^{-1}$ and $-\nabla\nabla\cdot (-\Delta)^{-1}$ are well defined, bounded, linear operators mapping $H^s \to H^s$, for all $s\in\mathbb{R}$, and so the decomposition is still valid and well defined for $s\leq \frac{d}{2}$.

In order to prove that the identity 
\begin{equation}
    v=-2\divr\nabla_{asym}(-\Delta)^{-1}v
    -\nabla \nabla\cdot(-\Delta)^{-1}v
\end{equation}
holds in $H^s_{as}$ for all $s\in\mathbb{R}$, we will need to work in Fourier space.
We will show that
\begin{align}
    \mathcal{F}\left(-2\divr\nabla_{asym}
    (-\Delta)^{-1}v\right)(\xi)
    &=
    P_{\spn\{\xi\}^\perp}\hat{v}(\xi) \\
    \mathcal{F}\left(-\nabla \nabla\cdot(-\Delta)^{-1}v\right)(\xi)
    &=
    P_{\spn\{\xi\}}\hat{v}(\xi),
\end{align}
almost everywhere $\xi\in\mathbb{R}^d$.
This establishes the decomposition, because it proves both the identity \eqref{HelmholtzID} and also shows that $-2\divr\nabla_{asym}
(-\Delta)^{-1}v\in H^s_*,
-\nabla \nabla\cdot(-\Delta)^{-1}v
\in H^s_{as\&gr}$.
It is simple to compute that
\begin{align}
    \mathcal{F}\left(-\nabla \nabla\cdot(-\Delta)^{-1}v\right)(\xi)
    &=
    \left(\frac{\xi\otimes\xi}{|\xi|^2}\right)
    \hat{v}(\xi) \\
    &=
    P_{\spn\{\xi\}}\hat{v}(\xi),
\end{align}
almost everywhere $\xi\in\mathbb{R}^d$.
We also compute that
\begin{equation}
    \mathcal{F}\left(-2\divr\nabla_{asym}
    (-\Delta)^{-1}v\right)(\xi)
    =\frac{1}{|\xi|^2} \xi\cdot (\xi\otimes \hat{v}(\xi)-
    \hat{v}(\xi)\otimes \xi).
\end{equation}
For all $\xi\in\mathbb{R}^d$, let
\begin{align}
    P_{\spn\{\xi\}^\perp}\hat{v}(\xi)
    &=w(\xi) \\
    P_{\spn\{\xi\}}\hat{v}(\xi)
    &= \lambda(\xi)\xi.
\end{align}
We can immediately see that almost everywhere 
$\xi\in\mathbb{R}^d$,
\begin{equation}
    \xi\otimes \hat{v}(\xi)-
    \hat{v}(\xi)\otimes \xi
    =
    \xi\otimes w(\xi)-
    w(\xi)\otimes \xi.
\end{equation}
By definition, we have $\xi\cdot w(\xi)=0$, almost everywhere $\xi\in\mathbb{R}^d$, and so we can conclude that
\begin{align}
    \xi\cdot (\xi\otimes \hat{v}(\xi)-
    \hat{v}(\xi)\otimes \xi)
    &=
    \xi\cdot (\xi\otimes w(\xi)-
    w(\xi)\otimes \xi) \\
    &=
    |\xi|^2 w(\xi),
\end{align}
and that therefore
\begin{equation}
    \mathcal{F}\left(-2\divr\nabla_{asym}
    (-\Delta)^{-1}v\right)(\xi)
    =
    P_{\spn\{\xi\}^\perp}\hat{v}(\xi).
\end{equation}
This completes the proof.
\end{proof}

\begin{remark}
Note that the above proposition shows that for all $v\in H^s_{as}, 
P_{df}(v)\in H^s_{as}.$
\end{remark}

\subsection{Local existence of solutions}

\begin{theorem} \label{EulerExistence}
For all $u^0\in H^s_{df}\left(\mathbb{R}^d\right), s> 1+\frac{d}{2}, d\geq 3$,
there exists a unique solution of the Euler equation $u\in C\left([0,T_{max}),H^s_{df}\left(
\mathbb{R}^d\right)\right)
\cap C^1\left([0,T_{max}),
H^{s-1}_{df}
\left(\mathbb{R}^d\right)\right)$,
where
\begin{equation}
    T_{max}>\frac{C_{s,d}}
    {\left\|u^0\right\|_{H^s}}.
\end{equation}
Furthermore, if $T_{max}<+\infty$,
\begin{equation}
    \int_0^{T_{max}}
    \|A(\cdot,t)\|_{L^\infty} \diff t
    =+\infty.
\end{equation}
\end{theorem}

\begin{remark}
This is the classic Beale-Kato-Majda criterion, and was proven in the case $d=3, s\geq 3$ in \cite{BKM}.
Kato and Ponce proved the general case where $d\geq 3, s>1+\frac{d}{2}$ in \cite{KatoPonce}, 
They also proved an analogous regularity criterion in terms of the strain, showing that if $T_{max}<+\infty,$ then
\begin{equation}
    \int_0^{T_{max}}
    \|S(\cdot,t)\|_{L^\infty}
    \diff t=+\infty.
\end{equation}
See also the work of Kato \cite{KatoLocalWP} for details on the local wellposedness theory. For a thorough overview, see section 3.2 in Majda and Bertozzi \cite{MajdaBertozzi}. Note that this gives a definition of finite-time blowup that 
\begin{equation}
    \lim_{t\to T_{max}}\|u(\cdot,t)\|_{H^s}
    =+\infty,
\end{equation}
as it is clear that $T_{max}<+\infty$ implies
\begin{equation}
    \|u(\cdot,t)\|_{H^s}
    \geq \frac{C_{s,d}}{T_{max}-t}.
\end{equation}
\end{remark}

\begin{theorem} \label{AxisymExistence}
For all $u^0\in H^s_*\left(\mathbb{R}^d\right), s> 1+\frac{d}{2}$,
there exists a unique solution of the Euler equation 
$u\in C\left([0,T_{max}),H^s_*\left(
\mathbb{R}^d\right)\right)
\cap C^1\left([0,T_{max}),H^{s-1}_*
\left(\mathbb{R}^d\right)\right).$
Furthermore, if $T_{max}<+\infty$, then
\begin{equation} \label{RegCritEuler}
    \int_0^{T_{max}}
    \|\omega(\cdot,t)\|_{L^\infty} \diff t
    =+\infty.
\end{equation}
\end{theorem}
\begin{proof}
We know from Theorem \ref{EulerExistence} that there must be a smooth solution of the Euler equation
$u \in C\left([0,T_{max}),H^s_{df}\left(
\mathbb{R}^d\right)\right) \cap
C^1\left([0,T_{max}),H^{s-1}_{df}\left(
\mathbb{R}^d\right)\right) $.
Applying Propositions \ref{AxisymAdvect} and \ref{AxisymHelmholtz}, we can see that if $u(\cdot,t)\in H^s_*$, then
$P_{df}((u\cdot\nabla)u)(\cdot,t)
\in H^{s-1}_*$.
Recalling that $\partial_tu+P_{df}((u\cdot\nabla)u)=0$,
we can conclude that if $u(\cdot,t)\in H^s_*$, then
$\partial_t u(\cdot,t)\in H^{s-1}_*$.
This implies the subspace of axisymmetric, swirl-free vector fields is preserved by the dynamics of the Euler equation, although this deals only with the formal level.

To make this rigourous, and not only formal, it is necessary to consider how the local existence for the Euler equation---Theorem \ref{EulerExistence}---is proven. In section 3.2 of \cite{MajdaBertozzi}, this local existence theorem is proven, first by proving local existence for a mollified Euler equation
\begin{equation}
    \partial_t u^\epsilon+
    P_{df}\mathcal{J}_\epsilon \left(\mathcal{J}_\epsilon u^\epsilon\cdot\nabla \mathcal{J}_\epsilon u^\epsilon \right)=0,
\end{equation}
and then by proving strong convergence $u^\epsilon \to u$ in the space $C\left([0,T_{max}),H^s_{df}\left(
\mathbb{R}^d\right)\right) 
\cap C^1\left([0,T_{max}),
H^{s-1}_{df}
\left(\mathbb{R}^d\right)\right)$.
Note that $\mathcal{J}_\epsilon$ is a regularization operator given by convolution with a smooth, positive, compactly supported function of mass $1$.
As long as $\mathcal{J}_\epsilon$ involves convolution with a radial function, then it will preserve the class of axisymmetric, swirl-free vector fields. The Picard iteration then guarantees that $u^\epsilon$ must be axisymmetric and swirl-free. Finally, $u$ must be axisymmetric and swirl-free, because it is trivial that this property is preserved by strong convergence in $H^s$.

Finally, we observe that
\begin{equation}
    \|A\|_{L^\infty}
    =
    \frac{1}{\sqrt{2}}
    \|\omega\|_{L^\infty},
\end{equation}
and this completes the proof.
\end{proof} 

\subsection{The vorticity equation}

Now that the necessary preliminaries are out of the way, we can derive the evolution equation for the vorticity in four and higher dimensions.

\begin{proposition} \label{VortEqProp}
Suppose $u\in C\left([0,T_{max});
H^s_* \left(\mathbb{R}^d\right)\right)
\cap C^1\left([0,T_{max}),H^{s-1}_*
\left(\mathbb{R}^d\right)\right), 
d\geq 3, s>2+\frac{d}{2}$, is a solution of the Euler equation. Then for all 
$0\leq t<T_{max}$, the vorticity satisfies the evolution equation
\begin{equation}
    \partial_t\omega
    +(u\cdot\nabla)\omega
    -k\frac{u_r}{r}\omega=0.
\end{equation}
\end{proposition}

\begin{proof}
We begin by observing that due to Sobolev embedding, we have $u(\cdot,t)\in C^2$ and consequently $\omega(\cdot,t)\in C^1$,
for all $0\leq t<T_{max}$.
Next we write the axisymmetric Euler equation in terms of the components $u_r$ and $u_z$, yielding
\begin{align}
    \partial_t u_r
    +(u\cdot\nabla)u_r
    +\partial_r p
    &=0 \\
    \partial_t u_z
    +(u\cdot\nabla)u_z
    +\partial_z p
    &=0.
    \end{align}
Recalling that 
$\omega=\partial_r u_z-\partial_z u_r$, 
we differentiate the equation for $u_r$ with respect to $z$ and the equation for $u_z$ with respect to $r$,
yielding
\begin{align} \label{VortDeriveA}
    \partial_t \partial_z u_r
    +(u\cdot\nabla)\partial_z u_r
    +(\partial_z u\cdot\nabla) u_r
    +\partial_r \partial_z p
    &=0 \\
    \label{VortDeriveB}
    \partial_t \partial_r u_z 
    +(u\cdot\nabla)\partial_r u_z
    +(\partial_r u\cdot\nabla)u_z
    +\partial_r \partial_z p
    &=0.
    \end{align}
Subtracting \eqref{VortDeriveA} from \eqref{VortDeriveB}, we find that
\begin{equation}
    \partial_t\omega +(u\cdot\nabla)\omega
    -(\partial_z u\cdot\nabla) u_r
    +(\partial_r u\cdot\nabla)u_z=0.
\end{equation}

It remains only to compute the difference term. Plugging in we find that
\begin{align}
    -(\partial_z u\cdot\nabla) u_r
    +(\partial_r u\cdot\nabla)u_z
    &=
    -\partial_z u_r \partial_r u_r
    -\partial_zu_z \partial_z u_r
    +\partial_r u_r \partial_r u_z
    +\partial_r u_z \partial_z u_z \\
    &= \label{CoordVortID}
    (\partial_ru_r+\partial_zu_z)
    (\partial_ru_z-\partial_zu_r) \\
    &=
    -k\frac{u_r}{r}\omega,
\end{align}
where we have used the divergence free constraint
\begin{equation}
    \partial_ru_r+\partial_zu_z
    +k\frac{u_r}{r}=0.
\end{equation}
This completes the proof.
\end{proof}

\begin{remark}
We will show in the next section that for axisymmetric, swirl-free vector fields, the velocity $u$ is uniquely determined by the the scalar vorticity $\omega$,
so the vorticity equation completely determines the dynamics of the axisymmetric, swirl-free Euler equation in four and higher dimensions, just as it does in three dimensions.
\end{remark}

One of the most important features of the axisymmetric Euler equation in three dimensions is that the quantity $\frac{\omega}{r}$ is advected by the flow, and therefore remains bounded in a number of important spaces. 
There is an analogous result in four and higher dimensions: the quantity $\frac{\omega}{r^k}$ is advected by the flow.

\begin{proposition} \label{AdvectedProp}
Suppose $u\in C\left([0,T_{max});
H^s_* \left(\mathbb{R}^d\right)\right)
\cap C^1\left([0,T_{max}),H^{s-1}_*
\left(\mathbb{R}^d\right)\right), 
d\geq 3, s>2+\frac{d}{2}$, is a solution of the Euler equation.
Then for all 
$0\leq t<T_{max}$,
\begin{equation}
    \left(\partial_t+u\cdot\nabla\right)
    \frac{\omega}{r^k}=0.
\end{equation}
If we let $\xi=\frac{\omega}{r^k}$,
then this can be expressed as
\begin{equation} \label{AdvectedEquation}
    \left(\partial_t+u\cdot\nabla\right)
    \xi =0.
\end{equation}
\end{proposition}

\begin{proof}
Applying the product rule and plugging into the vorticity equation we can see that
\begin{align}
    \left(\partial_t+u\cdot\nabla\right)
    \frac{\omega}{r^k}
    &=
    \frac{1}{r^k}\partial_t\omega
    +\frac{1}{r^k}(u\cdot\nabla)\omega
    -u_r k\frac{\omega}{r^{k+1}} \\
    &=
    \frac{1}{r^k}\left(\partial_t\omega
    +(u\cdot\nabla)\omega
    -k\frac{u_r}{r}\omega \right) \\
    &=0.
\end{align}
This completes the proof.
\end{proof}

\begin{proposition} \label{LagrangeMap}
Suppose $u\in C\left([0,T_{max});
H^s_* \left(\mathbb{R}^d\right)\right)
\cap C^1\left([0,T_{max}),H^{s-1}_*
\left(\mathbb{R}^d\right)\right), 
d\geq 3, s>2+\frac{d}{2}$, is a solution of the Euler equation,
and let $X(r,z,t)$ be the associated Lagrangian flow map satisfying
\begin{align}
    \frac{\diff X}{\diff t}(r,z,t)&=u(X(r,z,t),t) \\
    X(r,z,0)&=re_r+ze_z.
\end{align}
Then for all $0\leq t<T_{max}, r\in\mathbb{R}^+, z\in\mathbb{R}$,
\begin{equation}
    \frac{\omega(X(r,z,t),t)}{X_r(r,z,t)^k}
    =\frac{\omega^0(r,z)}{r^k}
\end{equation}
\end{proposition}

\begin{proof}
We will begin by letting
\begin{equation}
    \xi(r,z,t)=\frac{\omega(r,z,t)}{r^k}.
\end{equation}
We have already shown that $\xi$ is advected by the flow and so
\begin{equation}
    (\partial_t +u\cdot\nabla)\xi=0.
\end{equation}
Define $\Tilde{\xi}$ by
\begin{equation}
    \Tilde{\xi}(r,z,t)=\xi(X(r,z,t),t).
\end{equation}
Applying the chain rule, we can see that for all $0\leq t<T_{max}, r\in\mathbb{R}^+, z\in\mathbb{R}$,
\begin{align}
    \partial_t\Tilde{\xi}(r,z,t)
    &=
    (\partial_t\xi +\frac{\diff X}{\diff t}\cdot \nabla \xi)(X(r,z,t),t) \\
    &=
    (\partial_t\xi+u\cdot\nabla\xi)
    (X(r,z,t),t) \\
    &=0.
\end{align}
Therefore, we can conclude that 
for all $0\leq t<T_{max}, r\in\mathbb{R}^+, z\in\mathbb{R}$,
\begin{align}
    \xi(X(r,z,t),t)&=\xi(X(r,z,0),0) \\
    &=\xi^0(r,z),
\end{align}
and that consequently
\begin{equation}
    \frac{\omega(X(r,z,t),t)}{X_r(r,z,t)^k}
    =\frac{\omega^0(r,z)}{r^k}.
\end{equation}
This completes the proof.
\end{proof}

\begin{corollary}
Suppose $u\in C\left([0,T_{max});
H^s_* \left(\mathbb{R}^d\right)\right)
\cap C^1\left([0,T_{max}),H^{s-1}_*
\left(\mathbb{R}^d\right)\right), 
d\geq 3, s>2+\frac{d}{2}$, is a solution of the Euler equation, and that $\frac{\omega^0}{r^k}\in L^\infty$. Then for all $r\in\mathbb{R}^+,
z\in\mathbb{R},
0\leq t<T_{max}$,
\begin{equation}
    |\omega(r,z,t)| 
    \leq
    \left\|\frac{\omega^0}{r^k}
    \right\|_{L^\infty} r^k.
\end{equation}
\end{corollary}

\begin{proof}
We can see immediately from Proposition \ref{AdvectedProp} that for all $0\leq t<T_{max}$,
\begin{equation}
    \|\xi(\cdot,t)\|_{L^\infty}
    =
    \left\|\xi^0\right\|_{L^\infty},
\end{equation}
and therefore for all 
$r\in\mathbb{R}^+,
z\in\mathbb{R},
0\leq t<T_{max}$,
\begin{equation}
    \frac{\omega(r,z,t)}{r^k}
    \leq 
    \left\|\xi^0\right\|_{L^\infty}.
\end{equation}
This completes the proof.
\end{proof}

\begin{corollary} \label{LorentzNorm}
Suppose $u\in C\left([0,T_{max});
H^s_* \left(\mathbb{R}^d\right)\right)
\cap C^1\left([0,T_{max}),H^{s-1}_*
\left(\mathbb{R}^d\right)\right), 
d\geq 3, s>2+\frac{d}{2}$, is a solution of the Euler equation, then the Lorentz quasinorms of $\frac{\omega}{r^k}$ are all preserved by the flow:
for all $1\leq p,q\leq +\infty$ and for all $0\leq t<T_{max}$,
\begin{equation}
    \left\|\frac{\omega}{r^k}(\cdot,t)
    \right\|_{L^{p,q}
    \left(\mathbb{R}^d\right)}
    =
    \left\|\frac{\omega^0}{r^k}
    \right\|_{L^{p,q}
    \left(\mathbb{R}^d\right)}
\end{equation}
\end{corollary} 
\begin{proof}
First we observe that by Proposition \ref{LagrangeMap}, for all $0\leq t<T_{max}$,
\begin{equation}
    \left\|\xi \circ X (\cdot,t)
    \right\|_{L^{p,q}
    \left(\mathbb{R}^d\right)}
    =
    \left\|\xi^0 \right\|_{L^{p,q}
    \left(\mathbb{R}^d\right)}.
\end{equation}
Because the transport velocity is divergence free the flow map $X(\cdot,t)$ is volume preserving and therefore,
\begin{equation}
    \left\|\xi \circ X (\cdot,t)
    \right\|_{L^{p,q}
    \left(\mathbb{R}^d\right)}
    =
    \left\|\xi (\cdot,t)
    \right\|_{L^{p,q}
    \left(\mathbb{R}^d\right)},
\end{equation}
and this completes the proof.
\end{proof}

\begin{remark}
We will note that there is a fundamental difference between the case $d=3$ and the case $d\geq 4$.
When we take $d=3$, must have $\frac{\omega^0}{r}\in L^\infty$ for sufficiently smooth initial data, while in $d\geq 4$, there are Schwartz class initial data where $\frac{\omega^0}{r^k}\notin L^\infty$.
We will show this now.
\end{remark}

\begin{proposition}
Suppose $u\in H^s_*
\left(\mathbb{R}^d\right),
s>2+\frac{d}{2}$.
Then $\frac{\omega}{r}\in L^\infty$
and
\begin{equation}
    \left\|\frac{\omega}{r}
    \right\|_{L^\infty}
    \leq
    C_{s,d}\|u\|_{H^s}.
\end{equation}
\end{proposition}

\begin{proof}
We will begin by observing that if 
$u\in H^s\left(\mathbb{R}^d\right)$, then clearly
$\nabla A \in H^{s-2}\left(\mathbb{R}^d\right),
s-2> \frac{d}{2}$, so by Sobolev embedding
\begin{align}
    \|\nabla A\|_{L^\infty}
    &\leq 
    C \|\nabla A\|_{H^{s-2}} \\
    &\leq 
    C \|u\|_{H^s}.
\end{align}
Now it remains only to control magnitude of $\frac{\omega}{r}$ by the magnitude of $\nabla A$.
By definition $\nabla A$ is a three tensor with
\begin{equation}
    (\nabla A)_{ijk}=
    \partial_i A_{jk}.
\end{equation}
Recall that 
\begin{equation}
    A(x)=\frac{1}{2}\omega(r,z)
    (e_r\otimes e_z-e_z\otimes e_r),
\end{equation}
and we can see that
\begin{align}
    \nabla A
    &=
    \frac{1}{2}\nabla\omega \otimes 
    (e_r\otimes e_z-e_z\otimes e_r)
    +
    \frac{1}{2}\omega
    (\nabla e_r\otimes e_z
    -\left( \nabla e_r\otimes e_z\right)^*) \\
    &= \label{GradAComp}
    \frac{1}{2}\nabla\omega \otimes 
    (e_r\otimes e_z-e_z\otimes e_r)
    + 
    \frac{1}{2}\frac{\omega}{r}
    \left(\Tilde{I}_d \otimes e_z
    -\left(\Tilde{I}_d \otimes e_z\right)^*\right),
\end{align}
where $M^*$ for a three tensor is transpose of the last two elements
\begin{equation}
    M^*_{ijk}=M_{ikj}.
\end{equation}
Recalling that $\nabla \omega=\partial_r\omega e_r+\partial_z\omega e_z$, and observing that the two terms in \eqref{GradAComp} are orthogonal, we can see that
\begin{equation}
    |\nabla A|^2=
    \frac{1}{2}|\nabla \omega|^2
    +\frac{d-2}{2}\left|
    \frac{\omega}{r}\right|^2.
\end{equation}
This clearly implies that
\begin{equation}
    \left\|\frac{\omega}{r}
    \right\|_{L^\infty}
    \leq 
    \sqrt{\frac{2}{d-2}}
    \|\nabla A\|_{L^\infty},
\end{equation}
and this completes the proof.
\end{proof}

\begin{proposition}
For all $d\geq 4$, there exists a Schwartz class divergence free vector field $u\in \mathcal{S}\left(\mathbb{R}^d\right)_{df}$ such that $\frac{\omega}{r^k}\notin L^\infty$.
In particular,
\begin{equation}
    u(x)=r(1-2z^2)\exp(-r^2-z^2) e_r
    -z(k+1-2r^2)\exp(-r^2-z^2) e_z,  
\end{equation}
satisfies these conditions.
\end{proposition}

\begin{proof}
We will begin by showing that $u$ is Schwartz class and divergence free.
Using that $x'=r e_r$, we can see that
\begin{equation}
    u(x)=\left( x'(1-2x_d^2)
    -x_d(k+1-2|x'|^2)e_d\right)
    \exp\left(-|x|^2\right),
\end{equation}
and so we can see that 
$u\in \mathcal{S}\left(\mathbb{R}^d\right)$.
Next we compute that
\begin{align}
    \nabla\cdot u
    &=
    \partial_r u_r+\frac{k}{r}u_r
    +\partial_zu_z \\
    &=
    \left((k+1-2r^2)(1-2z^2)
    +(-1+2z^2)(k+1-2r^2)
    \right)\exp(-r^2-z^2) \\
    &=
    0.
\end{align}
Now that we have shown that $u\in \mathcal{S}_{df}$, it remains only to show that $\frac{\omega}{r^k}$ is unbounded.
Computing the vorticity we find that
\begin{align}
    \omega&=
    \partial_r u_z-\partial_z u_r \\
    &=
    \left(4rz+2(k+1)rz -4r^3z
    +4rz+2rz-4rz^3\right)
    \exp(-r^2-z^2) \\
    &=
    rz\left(2k+12-4r^2-4z^2\right)
    \exp(-r^2-z^2),
\end{align}
and so we can conclude that
\begin{equation}
    \frac{\omega}{r^k}
    =
    \frac{z}{r^{d-3}}
    \left(2k+12-4r^2-4z^2\right)
    \exp(-r^2-z^2).
\end{equation}
By hypothesis, $d\geq 4$, and so
$\frac{\omega}{r^k} \notin L^\infty$,
and this completes the proof.
\end{proof}

We should note that the vorticity equation can also be expressed in ``divergence form'', by making use of the coordinate divergence
$\widetilde{\nabla}\cdot v=
    \partial_r v_r+\partial_z v_z$.

\begin{proposition}
Suppose $u\in C\left([0,T_{max});
H^s_* \left(\mathbb{R}^d\right)\right)
\cap C^1\left([0,T_{max}),H^{s-1}_*
\left(\mathbb{R}^d\right)\right), 
d\geq 3, s>2+\frac{d}{2}$, is a solution of the Euler equation.
Then for all $0\leq t<T_{max}$,
\begin{equation}
    \partial_t\omega 
    +\widetilde{\nabla}\cdot (\omega u)=0,
\end{equation}
that is
\begin{equation}
    \partial_t\omega
    +\partial_r (\omega u_r)
    +\partial_z (\omega u_z)
    =0.
\end{equation}
\end{proposition}

\begin{proof}
We begin by applying the product rule, and observing that
\begin{equation}
    \widetilde{\nabla}\cdot (\omega u)
    =
    (u\cdot\nabla)\omega
    +(\Tilde{\nabla}\cdot u)\omega.
\end{equation}
Recalling the identity \eqref{CoordVortID} from Proposition \ref{VortEqProp},
we already shown that 
\begin{equation}
    \partial_t\omega
    +(u\cdot\nabla)\omega
    +(\Tilde{\nabla}\cdot u)\omega
    =0,
\end{equation}
and this completes the proof.
\end{proof}

\section{The Biot-Savart law in higher dimensions} \label{BiotSavartSection}

The scalar vorticity uniquely determines the velocity for axisymmetric, swirl-free vector fields. In this section, we will derive integral formulae for recovering the velocity and velocity component functions from the vorticity.

\begin{proposition} \label{BiotSavart}
Suppose $u\in H^s_* \left(\mathbb{R}^d\right), s> \frac{d}{2}$. 
Then $u$ can be recovered from its scalar vorticity as
\begin{equation} \label{BS}
    u(x)=-\alpha_d \int_{\mathbb{R}^d} 
    \omega(\rb,\zb)( e_{\rb} \otimes e_z - e_z \otimes e_{\rb} ) \frac{x-\xb}{|x-\xb|^d} \diff \xb,
\end{equation}
where $\rb = |\xb'|, \zb = \xb_d$, and 
$\alpha_d=\frac{(d-2)\Gamma\left(
\frac{d}{2}-1\right)}{4\pi^\frac{d}{2}}$.
\end{proposition}

\begin{proof}
Recalling that  $A=\frac{1}{2}(\nabla u- (\nabla u)^{tr})$,
we can see that
\[
    -2\divr(A)=-\Delta u+ \nabla (\nabla \cdot u)
    =-\Delta u,
\]
and so $u$ can be obtained from $A$ by convolution with the Poisson kernel:
\[
\begin{split}
    u &= -2\divr (-\Delta)^{-1}A
    = -2\divr  \frac{\alpha_d}{d-2} \int_{\mathbb{R}^d}
    \frac{1}{|x-\xb|^{d-2}}A(\xb) \diff \xb \\
    &= -2\frac{\alpha_d}{d-2}\sum_{i=1}^d
    \int_{\mathbb{R}^d}
    \partial_{x_i}\frac{1}{|x-\xb|^{d-2}}A_{ij}(\xb)
    \diff \xb \\ &=
    -2\frac{\alpha_d}{d-2} \int_{\mathbb{R}^d}
    A^{tr}(\xb) \nabla_x \frac{1}{|x-\xb|^{d-2}} \\
    &=-\alpha_d
    \int_{\mathbb{R}^d}\omega(\rb,\zb)
    (e_{\rb} \otimes e_z - e_z \otimes e_{\rb}) 
    \frac{x-\xb}{|x-\xb|^d} \diff \xb.
\end{split}
\]
\end{proof}

\begin{lemma} \label{SphericalIntegration}
For any $f \in L^1([-1,1];\left(1-y_1^2\right)^{\frac{d-4}{2}} dy_1)$, 
\begin{align} \label{spherical}
    \int_{\S^{d-2}} f(y_1) \diff S(y)
    &=
    m_{d-3} \int_{-1}^1 
    \left(1-y_1^2\right)^{\frac{d-4}{2}} 
    f(y_1) \diff y_1 \\
    &= 
    m_{d-3} \int_0^\pi  f(\cos \th) \sin^{d-3} \th \diff \th,
\end{align}
where for $n \geq 0$,
$m_{n} =\frac{2\pi^{\frac{n+1}{2}}}
{\Gamma\left(\frac{n+1}{2}\right)}$
is the surface area of the $n$-sphere embedded in 
$\mathbb{R}^{n+1}$.
\end{lemma}

Note $m_0=2$, corresponding to the fact that $S^0$ has two points. The second line of \eqref{spherical} follows from
change of variables $y_1 = \cos \th$.

\begin{proof}
Write $\R^{d-1} \ni y=(y_1,\Tilde{y})$.
Observing that for $y\in \S^{d-2} \subset \R^{d-1}$,
\[
    y_1^2+|\Tilde{y}|^2=|y|^2=1,
\]
and also that
\[
    \diff S(y)= \frac{1}{\sqrt{1-y_1^2}}
    \diff y_1 \diff S(\Tilde{y}),
\]
we arrive at
\begin{align}
    \int_{\S^{d-2}}f(y_1) \diff S(y)
    &=
    \int_{-1}^1\int_{|\Tilde{y}|=\sqrt{1-y_1^2}}
    f(y_1) \frac{1}{\sqrt{1-y_1^2}} 
    \diff S(\Tilde{y}) \diff y_1 \\
    &= 
    \int_{-1}^1 f(y_1) \frac{1}{\sqrt{1-y_1^2}}
    \left|\S^{d-3}_{\sqrt{1-y_1^2}}\right|\diff y_1.  
\end{align}
The first equality in~\eqref{spherical} then follows from
\[
    \left|\S^{d-3}_{\sqrt{1-y_1^2}}\right|
    =m_{d-3}\left(\sqrt{1-y_1^2}\right)^{d-3},
\]
and the second follows from the substitution $y_1 = \cos \th$.
\end{proof}

\begin{proposition} 
\label{BiotSavartCoorindate1}
Suppose $u\in H^s_{df}\left(\mathbb{R}^d\right), s> \frac{d}{2}, d\geq 3$, is axisymmetric and swirl-free.
Then the components $u_r$ and $u_z$ can be recovered from the scalar vorticity using the formulas
\begin{multline} \label{radvelo}
    u_r(r,z)=
    -\frac{d-2}{2\pi}
    \int_{-\infty}^\infty \int_0^\infty 
    \rb^{d-2} (z-\zb) \omega(\rb,\zb)  \\
    \int_{-1}^1 \frac{y_1
    (1-y_1^2)^\frac{d-4}{2}}
    {\left(r^2+\rb^2-2r\rb y_1 
    +(z-\zb)^2\right)^\frac{d}{2}} 
    \diff y_1 \diff\rb \diff \zb
\end{multline} 
or
\begin{multline} \label{radvelo2}
    u_r(r,z)=-\frac{d-2}{2\pi}
    \int_{-\infty}^\infty \int_0^\infty 
    \rb^{d-2} (z-\zb) \omega(\rb,\zb) \\
    \int_{0}^\pi  \frac{\sin^{d-3} \th \cos \th}
    {\left(r^2+\rb^2-2r\rb \cos \th 
    +(z-\zb)^2\right)^\frac{d}{2}} 
    \diff\th \diff\rb \diff \zb
\end{multline}
and
\begin{equation} \label{horvelo}
    u_z(r,z)=\frac{d-2}{2\pi}
    \int_{-\infty}^\infty \int_0^\infty 
    \rb^{d-2} \omega(\rb,\zb)  
    \int_{-1}^1 \frac{(r y_1-\rb)
    (1-y_1^2)^\frac{d-4}{2}}
    {\left(r^2+\rb^2-2r\rb y_1 
    +(z-\zb)^2\right)^\frac{d}{2}} 
    \diff y_1 \diff\rb \diff \zb
\end{equation}
or
\begin{multline} \label{horvelo2}
    u_z(r,z)=\frac{d-2}{2\pi}
    \int_{-\infty}^\infty \int_0^\infty 
    \rb^{d-2} \omega(\rb,\zb) \\
    \int_{0}^\pi \frac{ \sin^{d-3} \th (r \cos \th -\rb)}
    {\left(r^2+\rb^2-2r\rb \cos \th 
    +(z-\zb)^2\right)^\frac{d}{2}} 
    \diff\th \diff\rb \diff \zb.
\end{multline}
\end{proposition}

\begin{proof}
We apply~\eqref{BS}. Using axisymmetry, we may set $x=re_1+ze_d$, and so $e_r=e_1$. Then
\[
\begin{split}
    u_r(r,z)
    &=
    e_1\cdot u(re_1+ze_d) \\
    &=
    -\alpha_d e_1\cdot
    \int_{\mathbb{R}^d} \omega(\rb,\zb)
    (e_{\rb} \otimes e_z -e_z \otimes e_{\rb})
    \frac{r e_1- \xb'+(z-\zb)e_z}
    {\left(r^2+|\xb'|^2- 2r \xb_1 
    +(z-\zb)^2\right)^\frac{d}{2}} \diff \xb\\
    &=
    -\alpha_d 
    \int_{\mathbb{R}^d} \omega(\rb,\zb) \frac{\xb_1}{\rb}
    \frac{z-\zb}{\left(r^2+\rb^2- 2r \xb_1 
    +(z-\zb)^2\right)^\frac{d}{2}} \diff \xb.
\end{split}
\]
Letting $y = \frac{\xb'}{\rb}$ parameterize 
$\S^{d-2} \subset \R^{d-1}$, we find that
\begin{multline}
    u_r(r,z)=-\alpha_d \int_{-\infty}^\infty
    \int_0^\infty
     \rb^{d-2} (z-\zb) \omega(\rb,\zb) \\
    \int_{\S^{d-2}} \frac{y_1}
    {\left(r^2+\rb^2-2r\rb y_1 
    +(z-\zb)^2\right)^\frac{d}{2}} 
    \diff S(y) \diff\rb \diff \zb.
\end{multline}
Expressions~\eqref{radvelo} and~\eqref{radvelo2} then follow from
applying Lemma~\ref{SphericalIntegration} on spherical integration, and observing that
\begin{equation}
    \alpha_d m_{d-3}=\frac{d-2}{2\pi},
\end{equation}
where $m_{d-3}$ is taken as in Lemma \ref{SphericalIntegration}.

The computation of $u_z$ is similar:
\[
\begin{split}
    u_z(r,z)
    &=
    e_d\cdot u(re_1+ze_d) \\
    &=
    -\alpha_d e_z\cdot
    \int_{\mathbb{R}^d} \omega(\rb,\zb)
    (e_{\rb}\otimes e_z -e_z \otimes e_{\rb})
    \frac{r e_1-\xb'+(z-\zb)e_z}
    {\left(r^2+\rb^2-2r x'_1 
    +(z-\zb)^2\right)^\frac{d}{2}} \diff \xb\\
    &=
    \alpha_d \int_{\mathbb{R}^d} \omega(\rb,\zb)
    \frac{r y_1-\rb}
    {\left(r^2+\rb^2-2r \rb y_1 
    +(z-\zb)^2\right)^\frac{d}{2}} \diff \xb\\
    &=
    \alpha_d \int_{-\infty}^\infty\int_0^\infty
    m_{d-2} \rb^{d-2} \omega(\rb,\zb)  
    \int_{\S^{d-2}} \frac{r y_1-\rb}
    {\left(r^2+\rb^2-2r\rb y_1 
    +(z-\zb)^2\right)^\frac{d}{2}} 
    \diff S(y) \diff\rb \diff \zb,
\end{split}
\]
and again applying Lemma~\ref{SphericalIntegration}
produces expressions~\eqref{horvelo} and~\eqref{horvelo2}.
\end{proof}

Next we extend an upper bound of~\cite{FengSverak} on the radial velocity from $d=3$ to all $d \geq 3$.

\begin{lemma} \label{VelocityBoundLemma}
   For all $d\geq 3$, define $F_d:[0,+\infty] \to (0,+\infty)$ by
   \begin{equation}
    F_d(s)=
    \int_0^\pi
    \frac{\th^{d-3}}{(1+2s(1-\cos(\theta))^\frac{d}{2}}  \diff\th.
   \end{equation}
   Then for all $s> 0$,
   \begin{equation} \label{LemmaBoundA}
    F_d(s)\leq \frac{\pi^{d-2}}{d-2},\quad 
    F_d(s)\leq \frac{C}{s^{\frac{d}{2}-1}},\quad
    C=
    \int_0^\infty
    \frac{\tau^{d-3}}{(1+\frac{4}{\pi^2}\tau^2)^\frac{d}{2}} \diff\tau.
   \end{equation}
\end{lemma}

\begin{proof}
    For the first inequality, observe that $1-\cos(\theta)\geq 0$, and so $F_d$ is clearly a decreasing function in $s$, and therefore
    \begin{equation*}
    F_d(s)\leq F_d(0)=\frac{\pi^{d-2}}{d-2}.
    \end{equation*}
    For the second inequality, observe that for all $0\leq \th \leq \pi,$ 
    we have $1-\cos(\theta)\geq \frac{2}{\pi^2}\th^2$. Therefore
    \begin{equation*}
    F_d(s)\leq 
    \int_0^\pi
    \frac{\th^{d-3}}{(1+\frac{4s}{\pi^2}\theta^2)^\frac{d}{2}}  \diff\th.
    \end{equation*}
    Now take the change of variables $\tau=\sqrt{s}\theta$, and we can see that
    \[
    F_d(s)
    \leq 
    \left(\frac{1}{\sqrt{s}}
    \right)^{d-2}
    \int_0^{\pi\sqrt{s}}
    \frac{\tau^{d-3}}{(1+\frac{4}{\pi^2}\tau^2)^\frac{d}{2}}  \diff\tau \leq 
    \frac{1}{s^{\frac{d}{2}-1}}
    \int_0^\infty
    \frac{\tau^{d-3}}{(1+\frac{4}{\pi^2}\tau^2)^\frac{d}{2}} \diff\tau.\qedhere
    \]
\end{proof}

\begin{proposition} \label{VeloInfinityNormBound}
Suppose $d\geq 3, \; s>\frac{d}{2}, \; 0 \leq b<\frac{1}{2},$ and $q=\frac{(d-2)b}{\frac{1}{2}-b}$. Then for all $u\in H^s_{df\&as}$, such that $\frac{\omega}{r^k}\in L^1\cap L^\infty$
and $r^q \om \in L^1$,
\begin{equation} \label{radvelobound2}
    \|u_r\|_{L^\infty}
    \leq C_{d}
    \left\|\frac{\omega}{r^k}
    \right\|_{L^\infty}^\frac{1}{2}
    \left\|\frac{\omega}{r^k}
    \right\|_{L^1}^b
    \left\|r^q \omega
    \right\|_{L^1}^{\frac{1}{2}-b},
\end{equation}
where $C_d$ depends only on $d$.
In particular, when $b=0$,
\begin{equation} \label{radvelobound1}
     \|u_r\|_{L^\infty}
    \leq C_d
    \left\|\frac{\omega}{r^k}
    \right\|_{L^\infty}^\frac{1}{2}
    \|\omega\|_{L^1}^{\frac{1}{2}}.
\end{equation}
\end{proposition}

\begin{proof}
We will prove~\eqref{radvelobound1}. Then~\eqref{radvelobound2} 
follows from writing
\[
  \om = \left( \frac{\om}{r^k} \right)^{2b} \left( r^q \om \right)^{1 - 2b}
\]
and applying H\"older's inequality. Note that the inequality in \eqref{radvelobound1} is invariant under translation in $z$, and under the rescaling (we leave this to the reader to check):
\begin{equation}
    u^\lambda(r,z)=u(\lambda r,\lambda z).
\end{equation}
Therefore it suffices to prove the inequality \eqref{radvelobound1} for $r=1, z=0$, and the result then follows for all $r>0, z\in\mathbb{R}$ by scaling and translation invariance (note that $u_r(0,z)=0$, so we get the case $r=0$ for free). Fixing $r=1, z=0$, we can apply Proposition \ref{BiotSavartCoorindate1} to observe
\begin{align*}
    u_r(1,0)
    &=
    \frac{d-2}{2\pi}
    \int_{-\infty}^\infty 
    \int_0^\infty 
    \int_{0}^\pi
    \rb^{d-2} \zb \omega(\rb,\zb)  
      \frac{\sin^{d-3} \th \cos \th}
    {\left(1+\rb^2-2\rb \cos \th 
    +\zb^2\right)^\frac{d}{2}} 
    \diff\th \diff\rb \diff \zb \\
    &=
    \frac{d-2}{2\pi}
    \int_{-\infty}^\infty 
    \int_0^\infty 
    \frac{\rb^{d-2} \zb \omega(\rb,\zb)}
    {\left((\rb-1)^2+\zb^2
    \right)^\frac{d}{2}} 
    \int_{0}^\pi
    \frac{\sin^{d-3} \th \cos \th}
    {\left(1+\frac{2\rb}
    {(\rb-1)^2+\zb^2}(1-\cos\theta)\right)^\frac{d}{2}} 
    \diff\th \diff\rb \diff \zb.
\end{align*}
Using the fact that $0\leq \sin(\theta)\leq \theta$ for all $0\leq \theta \leq \pi$, we can see that
\[
    |u_r(1,0)|
    \leq 
    \frac{d-2}{2\pi}
    \int_{-\infty}^\infty 
    \int_0^\infty 
    \frac{\rb^{d-2}|\omega(\rb,\zb)|}
    {\left((\rb-1)^2+\zb^2
    \right)^{\frac{d}{2}-\frac{1}{2}}}
    F_d\left(\frac{\rb}
    {(\rb-1)^2+\zb^2}\right)
    \diff\rb \diff \zb 
\] 
for $F_d$ defined in Lemma \ref{VelocityBoundLemma}.
Note that we have a singular integral kernel at $\rb=1, \zb=0$, but there will be some cancellation, because $s=\frac{2\rb}{(\rb-1)^2+\zb^2} \to +\infty$ as $(\rb,\zb)\to (1,0)$, and $F_d(s)\leq \frac{C}{s^{\frac{d}{2}-1}}$. 
We will now break the domain into three pieces to deal with the singular integral. We have
\begin{equation*}
    |u_r(1,0)|\leq \frac{d-2}{2\pi} \sum_{j=1}^3 I_j,\quad
    I_j
    =
    \iint\limits_{A_j}
    \frac{\rb^k|\omega(\rb,\zb)|}
    {\left((\rb-1)^2+\zb^2
    \right)^{\frac{k}{2}+\frac{1}{2}}}
    F_d\left(\frac{\rb}
    {(\rb-1)^2+\zb^2}\right)
    \diff\rb \diff \zb,
\end{equation*}
where
\[
A_1=\big\{ (\rb-1)^2+\zb^2\leq \tfrac{1}{4}\big\},\quad
A_2=\big\{\tfrac{1}{4} \leq (\rb-1)^2+\zb^2\leq 4\big\},\quad
A_3=\big\{4 \leq (\rb-1)^2+\zb^2\big\}.
\]

We begin by bounding $I_2$. Applying \eqref{LemmaBoundA} from Lemma \ref{VelocityBoundLemma}, we find that
\begin{align}
    I_2
    &\leq 
    C\iint\limits_{A_2}
    \frac{\rb^\frac{k}{2}|\omega(\rb,\zb)|}
    {\left((\rb-1)^2+\zb^2
    \right)^{\frac{1}{2}}}
    \diff\rb \diff \zb \leq 
    C \left\|\frac{\omega}{r^k}
    \right\|_{L^\infty}^\frac{1}{2}
    \iint\limits_{A_2}
    \frac{|\omega(\rb,\zb)|^\frac{1}{2}}
    {\left((\rb-1)^2+\zb^2
    \right)^{\frac{1}{2}}}
    \rb^k \diff\rb \diff \zb \\
    &\leq  \label{ExampleStep}
    C \left\|\frac{\omega}{r^k}
    \right\|_{L^\infty}^\frac{1}{2}
    \bigg(\;\iint\limits_{A_2}
    |\omega(\rb,\zb)|
    \rb^k \diff\rb \diff \zb\bigg)^\frac{1}{2}
    \bigg(\:\iint\limits_{A_2}
    \frac{1}{(\rb-1)^2+\zb^2}
    \rb^k \diff\rb \diff \zb
    \bigg)^\frac{1}{2} \\
    &\leq 
    C\left\|\frac{\omega}{r^k}
    \right\|^\frac{1}{2}
    \|\omega\|_{L^1}^\frac{1}{2},
\end{align}
where we have applied H\"older's inequality in the last step with respect to the measure $\rb^k\diff\rb\diff\zb$.

Now we will bound $I_1$, considering two cases. First suppose that 
$\left\|\frac{\omega}{r^k}
\right\|_{L^\infty}
\leq \|\omega\|_{L^1}$.
Then again applying Lemma \ref{VelocityBoundLemma} we find that
\begin{align*}
    I_1 
    &\leq 
    C\iint\limits_{(\rb-1)^2+\zb^2
    \leq \frac{1}{4}}
    \frac{\rb^\frac{k}{2}|\omega(\rb,\zb)|}
    {\left((\rb-1)^2+\zb^2
    \right)^{\frac{1}{2}}}
    \diff\rb \diff \zb \\
    &\leq 
    C\left\|\frac{\omega}{r^k}
    \right\|_{L^\infty}
    \iint\limits_{(\rb-1)^2+\zb^2\leq 
    \frac{1}{4}}
    \frac{1}
    {\left((\rb-1)^2+\zb^2
    \right)^{\frac{1}{2}}}
    \diff\rb \diff \zb \\
    &\leq
    C \left\|\frac{\omega}{r^k}
    \right\|_{L^\infty} 
    \leq 
    C\left\|\frac{\omega}{r^k}
    \right\|^\frac{1}{2}
    \|\omega\|_{L^1}^\frac{1}{2}.
\end{align*}
Now suppose that $\left\|\frac{\omega}{r^k}
\right\|_{L^\infty}
> \|\omega\|_{L^1}$,
and let $R=\frac12 {\|\omega\|_{L^1}^\frac{1}{2}}
\left\|\frac{\omega}
{r^k}\right\|_{L^\infty}^{-\frac{1}{2}}$.
Then
\begin{align*}
    I_1
    &\leq 
    \iint\limits_{(\rb-1)^2+\zb^2\leq R^2}
    \frac{\rb^\frac{k}{2}|\omega(\rb,\zb)|}
    {\left((\rb-1)^2+\zb^2
    \right)^{\frac{1}{2}}}
    \diff\rb \diff \zb
    +
    \iint\limits_{R^2\leq (\rb-1)^2+\zb^2\leq 
    \frac{1}{4}}
    \frac{\rb^\frac{k}{2}|\omega(\rb,\zb)|}
    {\left((\rb-1)^2+\zb^2
    \right)^{\frac{1}{2}}}
    \diff\rb \diff \zb \\
    &\leq 
    C \left\|\frac{\omega}{r^k}
    \right\|_{L^\infty}
    \iint\limits_{(\rb-1)^2+\zb^2\leq R^2}
    \frac{1}
    {\left((\rb-1)^2+\zb^2
    \right)^{\frac{1}{2}}}
    \diff\rb \diff \zb
    +\frac{C}{R}
    \iint\limits_{R^2\leq (\rb-1)^2+\zb^2\leq \frac{1}{4}}
    |\omega(\rb,\zb)|\rb^k 
    \diff\rb \diff\zb \\
    &\leq
    C R\left\|\frac{\omega}{r^k}
    \right\|_{L^\infty}
    +
    \frac{C}{R}
    \|\omega\|_{L^1} 
    =
    C \left\|\frac{\omega}{r^k}
    \right\|_{L^\infty}^\frac{1}{2}
    \|\omega\|_{L^1}^\frac{1}{2}.
\end{align*}

It remains now only to bound $I_3$. First suppose that 
$\|\omega\|_{L^1} \leq 
\left\|\frac{\omega}{r^k}
\right\|_{L^\infty}$.
Then we can see from the bound \eqref{LemmaBoundA} from Lemma \ref{VelocityBoundLemma} that
\begin{align*}
    I_3
    &\leq 
    C \iint\limits_{(\rb-1)^2+\zb^2 \geq 4}
    \frac{\rb^k|\omega(\rb,\zb)|}
    {\left((\rb-1)^2+\zb^2
    \right)^{\frac{k}{2}+\frac{1}{2}}}
    \diff\rb \diff \zb \\
    &\leq 
    C \iint\limits_{(\rb-1)^2+\zb^2 \geq 4}
    \rb^k|\omega(\rb,\zb)|
    \diff\rb \diff \zb 
    \leq 
    C\|\omega\|_{L^1} 
    \leq 
    C\|\omega\|_{L^1}^\frac{1}{2}
    \left\|\frac{\omega}{r^k}
    \right\|_{L^\infty}^\frac{1}{2}.
\end{align*}
Next suppose that $\|\omega\|_{L^1} >
\left\|\frac{\omega}{r^k}
\right\|_{L^\infty}$. Let 
$R=2\left({\|\omega\|_{L^1}}{\left\|\frac{\omega}{r^k}
\right\|^{-1}_{L^\infty}}\right)^\frac{1}{2(k+1)}$.
Then 
\begin{align*}
    I_3
    &\leq 
    C \iint\limits_{4\leq (\rb-1)^2+\zb^2 
    \leq R^2}
    \frac{\rb^k|\omega(\rb,\zb)|}
    {\left((\rb-1)^2+\zb^2
    \right)^{\frac{k}{2}+\frac{1}{2}}}
    \diff\rb \diff \zb
    + 
    C \iint\limits_{(\rb-1)^2+\zb^2 \geq R^2}
    \frac{\rb^k|\omega(\rb,\zb)|}
    {\left((\rb-1)^2+\zb^2
    \right)^{\frac{k}{2}+\frac{1}{2}}}
    \diff\rb \diff \zb \\
    &\leq 
    C \left\|\frac{\omega}{r^k}
    \right\|_{L^\infty}
    \iint\limits_{4\leq (\rb-1)^2+\zb^2 
    \leq R^2}
    \frac{\rb^{2k}}
    {\left((\rb-1)^2+\zb^2
    \right)^{\frac{k}{2}+\frac{1}{2}}}
    \diff\rb \diff \zb
    +\frac{C}{R^{k+1}}
    \iint\limits_{(\rb-1)^2+\zb^2 \geq R^2}
   |\omega(\rb,\zb)|
    \rb^k \diff\rb \diff \zb \\
    &\leq 
    CR^{k+1} 
    \left\|\frac{\omega}{r^k}
    \right\|_{L^\infty}
    +\frac{C}{R^{k+1}}\|\omega\|_{L^1} 
    =
    C \|\omega\|_{L^1}^\frac{1}{2}
    \left\|\frac{\omega}{r^k}
    \right\|_{L^\infty}^\frac{1}{2}.
\end{align*}

We have now shown, by bounding $I_1, I_2,$ and $I_3$,  that
\begin{equation}
    |u_r(1,0)|\leq 
    C_d \|\omega\|_{L^1}^\frac{1}{2}
    \left\|\frac{\omega}{r^k}
    \right\|_{L^\infty}^\frac{1}{2},
\end{equation}
which completes the proof. Note that the dependence of the constant $C_d$ on the dimension comes from the factor of $\frac{d-2}{2\pi}$ from the Biot-Savart law, the constants in Lemma \ref{VelocityBoundLemma} that are expressed explicitly in terms of $d$, as well as in the bounds on certain definite integrals (e.g. \eqref{ExampleStep}).
\end{proof}

\section{Conditional regularity in four and higher dimensions}
\label{RegCritSection}

In this section we will prove conditional regularity results in four and higher dimensions when the advected quantity, $\frac{\omega}{r^{d-2}}$, is bounded for the initial data.

\begin{proposition} \label{VortBoundProp}
Suppose $u\in C\left([0,T_{max}),H^s_{df}
\left(\mathbb{R}^d\right)\right)
\cap C^1\left([0,T_{max}),H^{s-1}_{df}
\left(\mathbb{R}^d\right)\right),
d\geq 3, s>2+\frac{d}{2}$, is an axisymmetric, swirl-free solution of the Euler equation,
and that $\frac{\omega^0}{r^k}\in L^\infty$.
Then for all $R>0$ and for all $0\leq t<T_{max}$,
\begin{equation}
    \|\omega(\cdot,t)\|_{L^\infty}
    \leq
    \max\left(\left\|\omega^0
    \right\|_{L^\infty(\mathcal{C}_R^c)},
    R^k \left\|\frac{\omega^0}{r^k}
    \right\|_{L^\infty(\mathcal{C}_R)}\right)
    \left(1+\frac{1}{R}\int_0^t
    \|u_r^+(\cdot,\tau)\|_{L^\infty}
    \diff \tau\right)^k.
\end{equation}
Recall that $\mathcal{C}_R$ is the cylinder
$\left\{(r,z): r<R\right\}$.
\end{proposition}

\begin{proof}
Again letting $X(r,z,t)$ be the associated flow map, and recalling Proposition \ref{LagrangeMap},
we let
\begin{equation}
    \omega(X(r,z,t),t)=
    \frac{\omega^0(r,z)}{r^k}X_r(r,z,t)^k.
\end{equation}
We know that
\begin{equation}
    X_r(r,z,t)=r+\int_0^t
    u_r(X(r,z,\tau),\tau) \diff\tau,
\end{equation}
and so we can conclude that for all
$0\leq t<T_{max}, r\in\mathbb{R}^+, z\in\mathbb{R}$,
\begin{equation}
    X_r(r,z,t)\leq r+\int_0^t 
    \|u_r^+(\cdot,\tau)\|_{L^\infty} \diff\tau.
\end{equation}
This clearly implies that
\begin{equation}
    |\omega(X(r,z,t),t)|
    \leq 
    \frac{|\omega^0(r,z)|}{r^k}
    \left(r+\int_0^t \|u_r^+(\cdot,\tau)
    \|_{L^\infty}\diff\tau\right)^k.
\end{equation}

Now fix $R>0$.
We can clearly see that for all 
$0\leq r\leq R$,
\begin{align}
    |\omega(X(r,z,t),t)|
    &\leq 
    \frac{|\omega^0(r,z)|}{r^k}
    \left(R+\int_0^t \|u_r^+(\cdot,\tau)
    \|_{L^\infty}\diff\tau\right)^k \\
    &=
    \frac{|\omega^0(r,z)|}{r^k}
     R^k \left(1+\frac{1}{R}\int_0^t \|u_r^+(\cdot,\tau)
    \|_{L^\infty}\diff\tau\right)^k.
\end{align}
Likewise, for all $r\geq R$,
\begin{align}
    |\omega(X(r,z,t),t)|
    &\leq 
    |\omega^0(r,z)|
    \left(1+\frac{1}{r}
    \int_0^t \|u_r^+(\cdot,\tau)\|_{L^\infty}
    \diff\tau \right)^k \\
    &\leq 
    |\omega^0(r,z)|
    \left(1+\frac{1}{R}\int_0^t 
    \|u_r^+(\cdot,\tau)\|_{L^\infty}
    \diff\tau \right)^k.
\end{align}
Therefore, we can clearly see that for all $0\leq t<T_{max}$,
\begin{equation}
    \|\omega \circ X(\cdot,t)\|_{L^\infty}
    \leq
    \max\left(\left\|\omega^0
    \right\|_{L^\infty(\mathcal{C}_R^c)},
    R^k\left\|\frac{\omega^0}{r^k}
    \right\|_{L^\infty(\mathcal{C}_R)}\right)
    \left(1+\frac{1}{R}\int_0^t
    \|u_r^+(\cdot,\tau)\|_{L^\infty}
    \diff \tau\right)^k.
\end{equation}
Recall that $X(\cdot,t)$ is a volume-preserving diffeomorphism, we see that
\begin{equation}
    \|\omega\circ X(\cdot,t)\|_{L^\infty}
    =\|\omega(\cdot,t)\|_{L^\infty},
\end{equation}
and this completes the proof.
\end{proof}

\begin{corollary}
Suppose $u\in C\left([0,T_{max}),H^s_{df}
\left(\mathbb{R}^d\right)\right)
\cap C^1\left([0,T_{max}),H^{s-1}_{df}
\left(\mathbb{R}^d\right)\right),
d\geq 4, s>2+\frac{d}{2}$ is an axisymmetric, swirl-free solution of the Euler equation, that $\frac{\omega^0}{r^k}\in L^\infty$.
If $T_{max}<+\infty$, then
\begin{equation}
    \int_0^{T_{max}}
    \|u_r^+(\cdot,t)\|_{L^\infty}
    \diff t=+\infty.
\end{equation}
\end{corollary}

\begin{proof}
This follows immediately from Proposition \ref{VortBoundProp} and the Beale-Kato-Majda theorem.
\end{proof}

\begin{proposition} \label{VortBoundPropQ}
Suppose $u\in C\left([0,T_{max}),H^s_{df}
\left(\mathbb{R}^d\right)\right)
\cap C^1\left([0,T_{max}),H^{s-1}_{df}
\left(\mathbb{R}^d\right)\right),
d\geq 3, s>2+\frac{d}{2}$, is an axisymmetric, swirl-free solution of the Euler equation,
and that $\frac{\omega^0}{r^k}\in L^\infty$.
Then for all $1\leq q<+\infty$, for all $R>0$,
and for all $0\leq t<T_{max}$,
\begin{equation}
    \|\omega(\cdot,t)\|_{L^q}
    \leq
    \left(\left\|\omega^0
    \right\|_{L^q(\mathcal{C}_R^c)}^q
    +R^{kq} \left\|\frac{\omega^0}{r^k}
    \right\|_{L^q(\mathcal{C}_R)}^q
    \right)^\frac{1}{q}
    \left(1+\frac{1}{R}\int_0^t
    \|u_r^+(\cdot,\tau)\|_{L^\infty}
    \diff \tau\right)^k.
\end{equation}
\end{proposition}

\begin{proof}
Begin by fixing $R>0$.
As we have already shown above,
for all $0\leq r\leq R$,
\begin{align}
    |\omega(X(r,z,t),t)|
    &\leq 
    \frac{|\omega^0(r,z)|}{r^k}
    \left(R+\int_0^t \|u_r^+(\cdot,\tau)
    \|_{L^\infty}\diff\tau\right)^k \\
    &=
    \frac{|\omega^0(r,z)|}{r^k}
     R^k \left(1+\frac{1}{R}\int_0^t \|u_r^+(\cdot,\tau)
    \|_{L^\infty}\diff\tau\right)^k,
\end{align}
and likewise, for all $r\geq R$,
\begin{align}
    |\omega(X(r,z,t),t)|
    &\leq 
    |\omega^0(r,z)|
    \left(1+\frac{1}{r}
    \int_0^t \|u_r^+(\cdot,\tau)\|_{L^\infty}
    \diff\tau \right)^k \\
    &\leq 
    |\omega^0(r,z)|
    \left(1+\frac{1}{R}\int_0^t 
    \|u_r^+(\cdot,\tau)\|_{L^\infty}
    \diff\tau \right)^k.
\end{align}
Therefore, we can clearly 
see that for all $0\leq t<T_{max}$,
\begin{equation}
    \|\omega\circ X(\cdot,t)\|_{L^q
    \left(\mathcal{C}_R\right)}
    \leq 
    \left\|\frac{\omega^0}{r^k}\right\|_{L^q
    \left(\mathcal{C}_R\right)}
     R^k \left(1+\frac{1}{R}\int_0^t \|u_r^+(\cdot,\tau)
    \|_{L^\infty}\diff\tau\right)^k,
\end{equation}
and that
\begin{equation}
    \|\omega\circ X(\cdot,t)\|_{L^q
    \left(\mathcal{C}_R^c\right)}
    \leq 
    \left\|\omega^0\right\|_{L^q
    \left(\mathcal{C}_R^c\right)}
     \left(1+\frac{1}{R}\int_0^t \|u_r^+(\cdot,\tau)
    \|_{L^\infty}\diff\tau\right)^k.
\end{equation}

We can compute that
\begin{align*}
    \|\omega\circ X(\cdot,t)\|_{L^q
    \left(\mathbb{R}^d\right)}^q
    &=
    \|\omega\circ X(\cdot,t)\|_{L^q
    \left(\mathcal{C}_R\right)}^q
    +\|\omega\circ X(\cdot,t)\|_{L^q
    \left(\mathcal{C}_R^c\right)}^q \\
    &\leq
    \left(\left\|\frac{\omega^0}{r^k}
    \right\|_{L^q
    \left(\mathcal{C}_R\right)}^q R^{kq} 
    +\left\|\omega^0\right\|_{L^q
    \left(\mathcal{C}_R^c\right)}^q
    \right)
    \left(1+\frac{1}{R}\int_0^t \|u_r^+(\cdot,\tau)
    \|_{L^\infty}\diff\tau\right)^{kq}
\end{align*}

Recall that $X(\cdot,t)$ is a volume-preserving diffeomorphism, we see that
\begin{equation}
    \|\omega\circ X(\cdot,t)\|_{L^q}
    =\|\omega(\cdot,t)\|_{L^q},
\end{equation}
and this completes the proof.
\end{proof}

\begin{proposition} \label{VeloBoundProp}
Suppose $u\in C\left([0,T_{max}),H^s_{df}
\left(\mathbb{R}^d\right)\right)
\cap C^1\left([0,T_{max}),H^{s-1}_{df}
\left(\mathbb{R}^d\right)\right),
d\geq 3, s>2+\frac{d}{2}$, is an axisymmetric, swirl-free solution of the Euler equation, and that
$\frac{\omega^0}{r^k} \in L^1 \cap L^\infty$.
Then for all $R>0$ and for all $0\leq t<T_{max}$
\begin{multline} \label{UmaxGrowth}
    \|u_r^+(\cdot,t)\|_{L^\infty}
    \leq C_d \left\|\frac{\omega^0}{r^k}
    \right\|_{L^\infty}^\frac{1}{2}
    \left(\left\|\omega^0
    \right\|_{L^{1}(\mathcal{C}_R^c)}
    +R^k\left\|\frac{\omega^0}{r^k}
    \right\|_{L^{1}(\mathcal{C}_R)}
    \right)^\frac{1}{2} \\
    \left(1+\frac{1}{R}\int_0^t
    \|u_r^+(\cdot,\tau)\|_{L^\infty} 
    \diff\tau \right)^\frac{k}{2}.
\end{multline}
Furthermore, letting
\begin{equation}
    f(t)=1+\frac{1}{R}\int_0^t
    \|u_r^+(\cdot,\tau)\|_{L^\infty} 
    \diff\tau,
\end{equation}
we can conclude that 
for all $0\leq t<T_{max}$,
\begin{equation}
    \frac{\diff f}{\diff t}
    \leq \mu f^\frac{k}{2},
\end{equation}
where
\begin{equation}
    \mu=\frac{C_d}{R} \left\|\frac{\omega^0}{r^k}
    \right\|_{L^\infty}^\frac{1}{2}
    \left(\left\|\omega^0
    \right\|_{L^{1}(\mathcal{C}_R^c)}
    +R^k\left\|\frac{\omega^0}{r^k}
    \right\|_{L^{1}(\mathcal{C}_R)}
    \right)^\frac{1}{2}.
\end{equation}
\end{proposition}

\begin{proof}
Fix $R>0$.
Applying Propositions \ref{VeloInfinityNormBound}
and \ref{VortBoundPropQ}, we find that 
\begin{align*}
    \|u_r^+(\cdot,t)\|_{L^\infty}
    &\leq C_d
    \left\|\frac{\omega(\cdot,t)}{r^k}
    \right\|_{L^\infty}^\frac{1}{2}
    \|\omega(\cdot,t)
    \|_{L^1}^{\frac{1}{2}} \\
    & \leq 
    \left\|\frac{\omega^0}{r^k}
    \right\|_{L^\infty}^\frac{1}{2}
    \left(\left\|\omega^0
    \right\|_{L^1(\mathcal{C}_R^c)}
    +R^k \left\|\frac{\omega^0}{r^k}
    \right\|_{L^1(\mathcal{C}_R)}
    \right)^\frac{1}{2}
    \left(1+\frac{1}{R}\int_0^t
    \|u_r^+(\cdot,\tau)\|_{L^\infty}
    \diff \tau\right)^\frac{k}{2}.
\end{align*}
Finally, we can use the fundamental theorem of calculus to compute that
\begin{align}
    \frac{\diff f}{\diff t}
    &=
    \frac{1}{R}\|u_r^+(\cdot,t)
    \|_{L^\infty} \\
    &\leq 
    \mu f^\frac{k}{2}.
\end{align}
This completes the proof.
\end{proof}

\begin{theorem} \label{4Dregularity}
Suppose the initial data $u^0\in H^s_{df}
\left(\mathbb{R}^4\right), s>4$
is axisymmetric and swirl-free and $\frac{\omega^0}{r^2}\in L^1\cap L^\infty$.
Then there exists a global smooth solution of the Euler equation $u\in C\left([0,+\infty),H^s_{df}
\left(\mathbb{R}^4\right)\right)
\cap C^1\left([0,+\infty),H^{s-1}_{df}
\left(\mathbb{R}^4\right)\right)$. Furthermore, we have a bound on vorticity,
with for all $R>0$ and for all $0\leq t<+\infty$,
\begin{equation} \label{ExpBound4}
    \|\omega(\cdot,t)\|_{L^\infty}
    \leq
    \max\left(\left\|\omega^0
    \right\|_{L^\infty(\mathcal{C}_R^c)},
    R^2 \left\|\frac{\omega^0}{r^2}
    \right\|_{L^\infty(\mathcal{C}_R)}\right)
    \exp(2\mu t),
\end{equation}
where
\begin{equation}
    \mu=\frac{C_d}{R} \left\|\frac{\omega^0}{r^k}
    \right\|_{L^\infty}^\frac{1}{2}
    \left(\left\|\omega^0
    \right\|_{L^{1}(\mathcal{C}_R^c)}
    +R^2\left\|\frac{\omega^0}{r^2}
    \right\|_{L^{1}(\mathcal{C}_R)}
    \right)^\frac{1}{2}.
\end{equation}
\end{theorem}

\begin{proof}
We know that such a solution must exist locally-in-time up until some maximal time $T_{max}$.
By the Beale-Kato-Majda criterion that if the bound \eqref{ExpBound4} holds for all times $0\leq t<T_{max}$, then we must have $T_{max}=+\infty$,
and therefore it suffices to prove this bound holds up until the maximal time of existence.

We will begin by fixing $R>0$ and letting
\begin{equation}
    f(t)=
    1+\frac{1}{R}\int_0^t
    \|u_r^+(\cdot,\tau)\|_{L^\infty} \diff\tau.
\end{equation}
We can see from Proposition \ref{VortBoundProp} that
\begin{equation}
    \|\omega(\cdot,t)\|_{L^\infty}
    \leq
    \max\left(\left\|\omega^0
    \right\|_{L^\infty(\mathcal{C}_R^c)},
    R^2 \left\|\frac{\omega^0}{r^2}
    \right\|_{L^\infty(\mathcal{C}_R)}\right)
    f(t)^2.
\end{equation}
We know that from Proposition \ref{VeloInfinityNormBound} that
\begin{equation}
    \frac{\diff f}{\diff t}\leq \mu f,
\end{equation}
and we can see by definition that $f(0)=1$.
Applying Gr\"onwall's inequality, we find that for all $0\leq t<T_{max}$,
\begin{equation}
    f(t)\leq \exp(\mu t),
\end{equation}
and consequently
\begin{equation}
    \|\omega(\cdot,t)\|_{L^\infty}
    \leq
    \max\left(\left\|\omega^0
    \right\|_{L^\infty(\mathcal{C}_R^c)},
    R^2 \left\|\frac{\omega^0}{r^2}
    \right\|_{L^\infty(\mathcal{C}_R)}\right)
    \exp(2\mu t).
\end{equation}
This completes the proof.
\end{proof}

\begin{theorem} \label{3Dregularity}
Suppose the initial data $u^0\in H^s_{df}
\left(\mathbb{R}^3\right), s>\frac{7}{2}$
is axisymmetric and swirl-free. Then there exists a global smooth solution of the Euler equation $u\in C\left([0,+\infty),H^s_{df}
\left(\mathbb{R}^3\right)\right)
\cap C^1\left([0,+\infty),H^{s-1}_{df}
\left(\mathbb{R}^3\right)\right)$. Furthermore, we have a bound on vorticity,
with for all $R>0$ and for all $0\leq t<+\infty$,
\begin{equation} \label{QuadBound3}
    \|\omega(\cdot,t)\|_{L^\infty}
    \leq
    \max\left(\left\|\omega^0
    \right\|_{L^\infty(\mathcal{C}_R^c)},
    R \left\|\frac{\omega^0}{r}
    \right\|_{L^\infty(\mathcal{C}_R)}\right)
    \left(1+\frac{1}{2}\mu t\right)^2,
\end{equation}
where
\begin{equation}
    \mu=\frac{C_d}{R} \left\|\frac{\omega^0}{r}
    \right\|_{L^\infty}^\frac{1}{2}
    \left(\left\|\omega^0
    \right\|_{L^{1}(\mathcal{C}_R^c)}
    +R\left\|\frac{\omega^0}{r}
    \right\|_{L^{1}(\mathcal{C}_R)}
    \right)^\frac{1}{2}.
\end{equation}
\end{theorem}

\begin{proof}
We know that such a solution must exist locally-in-time up until some maximal time $T_{max}$.
By the Beale-Kato-Majda criterion that if the bound \eqref{QuadBound3} holds for all times $0\leq t<T_{max}$, then we must have $T_{max}=+\infty$,
and therefore it suffices to prove this bound holds up until the maximal time of existence.

We will begin by fixing $R>0$ and letting
\begin{equation}
    f(t)=
    1+\frac{1}{R}\int_0^t
    \|u_r^+(\cdot,\tau)\|_{L^\infty} \diff\tau.
\end{equation}
We can see from Proposition \ref{VortBoundProp} that
\begin{equation}
    \|\omega(\cdot,t)\|_{L^\infty}
    \leq
    \max\left(\left\|\omega^0
    \right\|_{L^\infty(\mathcal{C}_R^c)},
    R \left\|\frac{\omega^0}{r}
    \right\|_{L^\infty(\mathcal{C}_R)}\right)
    f(t).
\end{equation}
We know that from Proposition \ref{VeloInfinityNormBound} that
\begin{equation}
    \frac{\diff f}{\diff t}\leq \mu f^\frac{1}{2},
\end{equation}
and we can see by definition that $f(0)=1$.
Integrating this differential inequality, we find that for all $0\leq t<T_{max}$,
\begin{equation}
    f(t)\leq \left(1+\frac{1}{2}\mu t\right)^2,
\end{equation}
and this completes the proof.
\end{proof}

\begin{theorem} \label{UpperBound5D+}
    Suppose $u\in C\left([0,T_{max}),H^s_{df}
\left(\mathbb{R}^d\right)\right)
\cap C^1\left([0,T_{max}),H^{s-1}_{df}
\left(\mathbb{R}^d\right)\right),
d\geq 5, s>2+\frac{d}{2}$, is an axisymmetric, swirl-free solution of the Euler equation,
and that $\frac{\omega^0}{r^k}
\in L^1\cap L^\infty$.
Then for all $R>0$ and for all 
$0\leq t<T_{max}$,
\begin{equation} \label{VortBound}
    \|\omega(\cdot,t)\|_{L^\infty}
    \leq
    \frac{\max\left(
    \left\|\omega^0\right\|_{L^\infty},
    R^k\left\|\frac{\omega^0}{r^k}
    \right\|_{L^\infty}\right)}
    {(1-\alpha t)^\frac{2(d-2)}{d-4}},
\end{equation}
where
\begin{equation}
    \alpha= \frac{(d-4)C_d}{2R} \left\|\frac{\omega^0}{r^k}
    \right\|_{L^\infty}^\frac{1}{2}
    \left(\left\|\omega^0
    \right\|_{L^{1}(\mathcal{C}_R^c)}
    +R^k\left\|\frac{\omega^0}{r^k}
    \right\|_{L^{1}(\mathcal{C}_R)}
    \right)^\frac{1}{2}.
\end{equation}
In particular, this implies that
\begin{equation}
    T_{max} \geq 
    \frac{2R}
    {(d-4)C_d \left\|\frac{\omega^0}{r^k}
    \right\|_{L^\infty}^\frac{1}{2}
    \left(\left\|\omega^0
    \right\|_{L^{1}(\mathcal{C}_R^c)}
    +R^k\left\|\frac{\omega^0}{r^k}
    \right\|_{L^{1}(\mathcal{C}_R)}
    \right)^\frac{1}{2}}.
\end{equation}
\end{theorem}

\begin{proof}
We know from the Beale-Kato-Majda theorem that if $T_{max}<+\infty$, then
\begin{equation}
    \int_0^{T_{max}}
    \|\omega(\cdot,t)\|_{L^\infty} \diff t
    =+\infty,
\end{equation}
and so it suffices to prove the bound \eqref{VortBound}, which immediately then implies the lower bound on $T_{max}$.
We will begin by defining $f$ as in Proposition \ref{VeloBoundProp}, letting
\begin{equation}
    f(t)=1+\frac{1}{R}\int_0^t\|u_r^+(\cdot,\tau)\|_{L^\infty}
    \diff\tau.
\end{equation}
Applying Proposition \ref{VortBoundProp}, we have for all $0\leq t<T_{max}$,
\begin{equation}
    \|\omega(\cdot,t)\|_{L^\infty}
    \leq 
    \max\left(
    \left\|\omega^0\right\|_{L^\infty},
    R^k\left\|\frac{\omega^0}{r^k}
    \right\|_{L^\infty}\right)f(t)^k,
\end{equation}
therefore, it suffices to show that
for all $0\leq t<T_{max}$,
\begin{equation}
    f(t)^k\leq 
    \frac{1}{(1-\alpha t)^\frac{2(d-2)}{d-4}}.
\end{equation}
Observing that $\alpha=\frac{d-4}{2}\mu$,
and applying Proposition \ref{VeloBoundProp},
we find that
\begin{equation}
    \frac{\diff f}{\diff t}
    \leq
    \frac{\alpha}{\frac{k}{2}-1} 
    f^\frac{k}{2},
\end{equation}
and that consequently
\begin{equation}
    \frac{\diff}{\diff t}
    \left(f^{1-\frac{k}{2}}\right)
    \geq
    -\alpha.
\end{equation}
Integrating this differential inequality and using the fact that $f(0)=1$ by definition, 
we find that for all $0\leq t<T_{max}$,
\begin{equation}
    f(t)^{1-\frac{k}{2}}\geq 1-\alpha t,
\end{equation}
and therefore
\begin{equation}
    f(t)^{\frac{1}{2}(d-4)}
    \leq 
    \frac{1}{1-\alpha t}.
\end{equation}
Finally, we observe that
for all $0\leq t<T_{max}$,
\begin{align}
    f(t)
    &\leq 
    \frac{1}{(1-\alpha t
    )^\frac{2}{d-4}} \\
    f(t)^k
    &\leq 
    \frac{1}{(1-\alpha t
    )^\frac{2(d-2)}{d-4}}.
\end{align}
This completes the proof.
\end{proof}

\begin{corollary} \label{LowerBoundCor}
Suppose $u\in C\left([0,T_{max}),H^s_{df}
\left(\mathbb{R}^d\right)\right)
\cap C^1\left([0,T_{max}),H^{s-1}_{df}
\left(\mathbb{R}^d\right)\right),
d\geq 5, s>2+\frac{d}{2}$ is an axisymmetric, swirl-free solution of the Euler equation,
with finite-time blowup at $T_{max}<+\infty$,
and that $\frac{\omega^0}{r^k}\in L^1\cap L^\infty$.
Then for all $R>0$ and for all 
$0\leq t<T_{max}$,
\begin{equation}
    \|\omega(\cdot,t)
    \|_{L^{1}(\mathcal{C}_R^c)}
    +R^k\left\|\frac{\omega^0}{r^k}
    \right\|_{L^1\left(\mathbb{R}^d\right)}
    \geq
    \frac{4R^2}{(d-4)^2 C_d^2
    \left\|\frac{\omega^0}{r^k}
    \right\|_{L^\infty\left(\mathbb{R}^d\right)}
    \left(T_{max}-t\right)^2}.
\end{equation}
Note in particular that this implies that
\begin{equation}
    \lim_{t\to T_{max}}
    \left(T_{max}-t\right)^2
    \|\omega(\cdot,t)\|_{L^{1}(\mathcal{C}_R^c)}
    \geq 
    \frac{4R^2}{(d-4)^2 C_d^2
    \left\|\frac{\omega^0}{r^k}
    \right\|_{L^\infty
    \left(\mathbb{R}^d\right)}}.
\end{equation}
\end{corollary}

\begin{proof}
Fix $R>0$.
Applying Theorem \ref{UpperBound5D+} to our solution at time $t$, we can see that for all $0\leq t<T_{max}$,
\begin{equation}
    T_{max}-t \geq 
    \frac{2R}
    {(d-4)C_d \left\|\frac{\omega(\cdot,t)}{r^k}
    \right\|_{L^\infty}^\frac{1}{2}
    \left(\left\|\omega
    \right(\cdot,t)\|_{L^{1}(\mathcal{C}_R^c)}
    +R^k\left\|\frac{\omega(\cdot,t)}{r^k}
    \right\|_{L^{1}(\mathcal{C}_R)}
    \right)^\frac{1}{2}}.
\end{equation}
Observe that for all $0\leq t<T_{max}$
\begin{equation}
    \left\|\frac{\omega(\cdot,t)}{r^k}
    \right\|_{L^\infty}
    =
    \left\|\frac{\omega^0}{r^k}
    \right\|_{L^\infty},
\end{equation}
and likewise
\begin{align}
    \left\|\frac{\omega(\cdot,t)}{r^k}
    \right\|_{L^{1}(\mathcal{C}_R)}
    &\leq
    \left\|\frac{\omega(\cdot,t)}{r^k}
    \right\|_{L^{1}\left(\mathbb{R}^d\right)} \\
    &=
    \left\|\frac{\omega^0}{r^k}
    \right\|_{L^{1}\left(\mathbb{R}^d\right)}.
\end{align}
Therefore we may conclude that for all 
$0\leq t<T_{max}$,
\begin{equation}
    T_{max}-t \geq 
    \frac{2R}
    {(d-4)C_d \left\|\frac{\omega^0}{r^k}
    \right\|_{L^\infty}^\frac{1}{2}
    \left(\left\|\omega
    \right(\cdot,t)\|_{L^{1}(\mathcal{C}_R^c)}
    +R^k\left\|\frac{\omega^0}{r^k}
    \right\|_{L^1}
    \right)^\frac{1}{2}}.
\end{equation}
Rearranging this inequality, we find that for all
$0\leq t<T_{max}$,
\begin{equation}
    \|\omega(\cdot,t)
    \|_{L^{1}(\mathcal{C}_R^c)}
    +R^k\left\|\frac{\omega^0}{r^k}
    \right\|_{L^1}
    \geq
    \frac{4R^2}{(d-4)^2 C_d^2
    \left\|\frac{\omega^0}{r^k}
    \right\|_{L^\infty}
    \left(T_{max}-t\right)^2},
\end{equation}
and this completes the proof.
\end{proof}

\begin{corollary} \label{SharpLowerBoundVort}
Suppose $u\in C\left([0,T_{max}),H^s_{df}
\left(\mathbb{R}^d\right)\right)
\cap C^1\left([0,T_{max}),H^{s-1}_{df}
\left(\mathbb{R}^d\right)\right),
d\geq 5, s>2+\frac{d}{2}$ is an axisymmetric, swirl-free solution of the Euler equation,
with finite-time blowup at $T_{max}<+\infty$,
and that $\frac{\omega^0}{r^k}\in L^1\cap L^\infty$.
Then for all $0\leq t<T_{max}$,
\begin{equation}
    \|\omega(\cdot,t)
    \|_{L^{1}\left(\mathbb{R}^d\right)}
    \geq 
    \left(\frac{M_d}
    {\left\|\frac{\omega^0}{r^k}
    \right\|_{L^\infty}^\frac{d-2}{d-4}
    \left\|\frac{\omega^0}{r^k}
    \right\|_{L^1}^\frac{2}{d-4}}\right)
    \frac{1}
    {\left(T_{max}-t\right)^{2\frac{d-2}{d-4}}},
\end{equation}
where
\begin{equation}
    M_d=\sup_{M>0}\left(\frac{4M^2}{(d-4)^2C_d^2}
    -M^k\right).
\end{equation}
\end{corollary}

\begin{proof}
For all $R>0$, we can see from Corollary \ref{LowerBoundCor}
\begin{align}
    \|\omega(\cdot,t)
    \|_{L^{1}\left(\mathbb{R}^d\right)}
    &\geq 
    \|\omega(\cdot,t)
    \|_{L^{1}(\mathcal{C}_R^c)} \\
    &\geq 
     \frac{4R^2}{(d-4)^2 C_d^2
    \left\|\frac{\omega^0}{r^k}
    \right\|_{L^\infty}
    \left(T_{max}-t\right)^2}
    -R^k\left\|\frac{\omega^0}{r^k}
    \right\|_{L^1}.
\end{align}
Now we will introduce the parameter
\begin{equation}
    M=\frac{R(T_{max}-t)^\frac{2}{d-4}}
    {\left\|\frac{\omega^0}{r^k}
    \right\|_{L^\infty}^\frac{1}{k-2}
    {\left\|\frac{\omega^0}{r^k}
    \right\|_{L^1}^\frac{1}{k-2}}}.
\end{equation}
Substituting in $M$, we can see that for all $M>0$
and for all $0\leq t<T_{max}$,
\begin{equation} \label{Mstep}
    \|\omega(\cdot,t)
    \|_{L^{1}\left(\mathbb{R}^d\right)}
    \geq 
    \left(\frac{4M^2}{(d-4)^2 C_d^2}
    -M^k\right)
    \left(\frac{1}{\left\|\frac{\omega^0}{r^k}
    \right\|_{L^\infty}^\frac{k}{k-2}
    \left\|\frac{\omega^0}{r^k}
    \right\|_{L^1}^\frac{2}{k-2}}\right)
    \frac{1}
    {\left(T_{max}-t\right)^{2\frac{d-2}{d-4}}}.
\end{equation}
We know that $k\geq 3$, and so we can observe that
\begin{equation}
    0<
    \sup_{M>0}
     \left(\frac{4M^2}{(d-4)^2 C_d^2}
    -M^k\right)
    <+\infty,
\end{equation}
and so taking the supremum over $M>0$ of \eqref{Mstep},
this completes the proof.
\end{proof}

\section{Anti-parallel vortex tubes} \label{AntiParallelTubesSection}

In this section, we will discuss the special case of colliding anti-parallel vortex tubes.
This scenario will involve vorticities which are odd with respect to $z$ and positive for $z>0$. We will begin by deriving a special form of the Biot-Savart law that applies to vorticities which are odd with respect to $z$.

\begin{proposition} \label{BiotSavartOdd}
Suppose $u\in H^s_{df}\left(\mathbb{R}^d\right), s> \frac{d}{2}, d\geq 3$, is axisymmetric and swirl-free, and that $\omega(r,z)$ is odd in $z$, with for all $r\geq 0, z\in\mathbb{R}$,
\begin{equation}
    \omega(r,-z)=-\omega(r,z)
\end{equation}
Then the velocity can be recovered from the vorticity in the upper half plane using the following formulas:
\begin{equation} \label{URkernel}
    u_r(r,z)= \frac{d-2}{2\pi}
    \int_0^\infty \int_0^\infty
    H(r,z,\bar{r},\bar{z}) \bar{r}^{d-2}
    \omega(\bar{r},\bar{z})
    \diff\bar{r} \diff \bar{z},
\end{equation}
where
\begin{multline}
    H(r,z,\bar{r},\bar{z})= \int_0^1
    \Bigg(
    \frac{z+\bar{z}}
    {\left(r^2+\bar{r}^2-2r\bar{r}\tau
    +(z+\bar{z})^2\right)^\frac{d}{2}} 
    +\frac{z-\bar{z}}
    {\left(r^2+\bar{r}^2+2r\bar{r}\tau
    +(z-\bar{z})^2\right)^\frac{d}{2}}  \\
    -\frac{z+\bar{z}}
    {\left(r^2+\bar{r}^2+2r\bar{r}\tau
    +(z+\bar{z})^2\right)^\frac{d}{2}} 
    -\frac{z-\bar{z}}
    {\left(r^2+\bar{r}^2-2r\bar{r}\tau
    +(z-\bar{z})^2\right)^\frac{d}{2}}
    \Bigg)
    \tau \left(1-\tau^2\right)^\frac{d-4}{2}
    \diff\tau,
\end{multline}
and
\begin{equation} \label{UZkernel}
    u_z(r,z)=-\frac{d-2}{2\pi}
    \int_0^\infty \int_0^\infty
    G(r,z,\bar{r},\bar{z})\bar{r}^{d-2}\omega(\bar{r},\bar{z})
    \diff\bar{r} \diff \bar{z},
\end{equation}
where
\begin{multline}
    G(r,z,\bar{r},\bar{z})= \int_0^1
    \Bigg(
    \frac{\bar{r}-r\tau}
    {\left(r^2+\bar{r}^2-2r\bar{r}\tau 
    +(z-\bar{z})^2\right)^\frac{d}{2}}
    +\frac{\bar{r}+r\tau}
    {\left(r^2+\bar{r}^2+2r\bar{r}\tau 
    +(z-\bar{z})^2\right)^\frac{d}{2}} \\
    -\frac{\bar{r}-r\tau}
    {\left(r^2+\bar{r}^2-2r\bar{r}\tau
    +(z+\bar{z})^2\right)^\frac{d}{2}}
    -\frac{\bar{r}+r\tau}
    {\left(r^2+\bar{r}^2+2r\bar{r}\tau
    +(z+\bar{z})^2\right)^\frac{d}{2}}
    \Bigg)
    \left(1-\tau^2\right)^\frac{d-4}{2}
    \diff\tau.
\end{multline}
Further note that $u_r(r,z)$ is even in $z$, while $u_z(r,z)$ is odd in $z$, and for all $r,z>0,$
\begin{align}
    u_r(0,z)&=0 \\
    u_z(r,0)&= 0.
\end{align}
\end{proposition}

\begin{proof}
We begin by recalling the Biot-Savart law in Proposition \ref{BiotSavartCoorindate1}
\begin{multline}
    u_r(r,z)=-\frac{d-2}{2\pi}
    \int_{-\infty}^\infty \int_0^\infty 
    \bar{r}^{d-2} (z-\bar{z}) \omega(\bar{r},\bar{z})  \\
    \int_{-1}^1 \frac{\Tilde{y}_1
    (1-\Tilde{y}_1^2)^\frac{d-4}{2}}
    {\left(r^2+\bar{r}^2-2r\bar{r} \Tilde{y}_1 
    +(z-\bar{z})^2\right)^\frac{d}{2}} 
    \diff\Tilde{y}_1 \diff\bar{r} \diff \bar{z}.
\end{multline}
Using the odd symmetry of $\omega$ in $\bar{z}$,
and the change of variables $\Tilde{z}=-\bar{z}$,
\begin{multline}
    -\int_{-\infty}^0 \int_0^\infty 
    \bar{r}^{d-2} (z-\bar{z}) \omega(\bar{r},\bar{z})  
    \int_{-1}^1 \frac{\Tilde{y}_1
    (1-\Tilde{y}_1^2)^\frac{d-4}{2}}
    {\left(r^2+\bar{r}^2-2r\bar{r} \Tilde{y}_1 
    +(z-\bar{z})^2\right)^\frac{d}{2}} 
    \diff\Tilde{y}_1 \diff\bar{r} \diff \bar{z}
    = \\
    \int_{0}^\infty \int_0^\infty 
    \bar{r}^{d-2} (z+\Tilde{z}) 
    \omega(\bar{r},\Tilde{z})  
    \int_{-1}^1 \frac{\Tilde{y}_1
    (1-\Tilde{y}_1^2)^\frac{d-4}{2}}
    {\left(r^2+\bar{r}^2-2r\bar{r} \Tilde{y}_1 
    +(z+\Tilde{z})^2\right)^\frac{d}{2}} 
    \diff\Tilde{y}_1 
    \diff\bar{r} \diff \Tilde{z}
\end{multline}
Therefore, the identity for $u_r$ can be rewritten as
\begin{multline}
    u_r(r,z)
    =
    \frac{d-2}{2\pi}
    \int_0^\infty\int_0^\infty
    \bar{r}^{d-2} \omega(\bar{r},\bar{z})
    \int_{-1}^1 \\ 
    \left(
    \frac{(z+\bar{z})\Tilde{y}_1
    (1-\Tilde{y}_1^2)^\frac{d-4}{2}}
    {\left(r^2+\bar{r}^2-2r\bar{r} \Tilde{y}_1 
    +(z+\bar{z})^2\right)^\frac{d}{2}} 
    -
    \frac{(z-\bar{z})\Tilde{y}_1
    (1-\Tilde{y}_1^2)^\frac{d-4}{2}}
    {\left(r^2+\bar{r}^2-2r\bar{r} \Tilde{y}_1 
    +(z-\bar{z})^2\right)^\frac{d}{2}} 
    \right)\diff\Tilde{y}_1 
    \diff\bar{r} \diff\bar{z}.
\end{multline}
Making the substitution $\tau=-\Tilde{y_1}$,
for $-1\leq \Tilde{y}\leq 0$,
we can see that 
\begin{multline}
    \int_{-1}^0 \left(
    \frac{(z+\bar{z})\Tilde{y}_1
    (1-\Tilde{y}_1^2)^\frac{d-4}{2}}
    {\left(r^2+\bar{r}^2-2r\bar{r} \Tilde{y}_1 
    +(z+\bar{z})^2\right)^\frac{d}{2}} 
    -
    \frac{(z-\bar{z})\Tilde{y}_1
    (1-\Tilde{y}_1^2)^\frac{d-4}{2}}
    {\left(r^2+\bar{r}^2-2r\bar{r} \Tilde{y}_1 
    +(z-\bar{z})^2\right)^\frac{d}{2}} 
    \right)\diff\Tilde{y}_1
    = \\
    \int_0^1 \left(
    -\frac{(z+\bar{z})\tau
    \left(1-\tau^2\right)^\frac{d-2}{4}}
    {\left(r^2+\bar{r}^2+2r\bar{r}\tau
    +(z+\bar{z})^2\right)\frac{d}{2}}
    +\frac{(z-\bar{z})\tau
    \left(1-\tau^2\right)^\frac{d-2}{4}}
    {\left(r^2+\bar{r}^2+2r\bar{r}\tau
    +(z-\bar{z})^2\right)\frac{d}{2}}
    \right) \diff\tau.
\end{multline}
Making the trivial substitution
$\tau=\Tilde{y}$ for
$0\leq\Tilde{y}\leq 1$, 
this completes the proof of the identity \eqref{URkernel}.

We proceed similarly for $u_z$, starting with the identity from Proposition \ref{BiotSavartCoorindate1},
\begin{equation*}
    u_z(r,z)=\frac{d-2}{2\pi}
    \int_{-\infty}^\infty \int_0^\infty 
    \bar{r}^{d-2} \omega(\bar{r},\bar{z})  
    \int_{-1}^1 \frac{(r \Tilde{y}_1-\bar{r})
    (1-\Tilde{y}_1^2)^\frac{d-4}{2}}
    {\left(r^2+\bar{r}^2-2r\bar{r} \Tilde{y}_1 
    +(z-\bar{z})^2\right)^\frac{d}{2}} 
    \diff\Tilde{y}_1 \diff\bar{r} \diff\bar{z}.
\end{equation*}
Again applying the odd symmetry of $\omega$ in $z$, we find that
\begin{multline}
    u_z(r,z)=-\frac{d-2}{2\pi}
    \int_0^\infty \int_0^\infty 
    \bar{r}^{d-2} \omega(\bar{r},\bar{z})  
    \int_{-1}^1 \\
    \left(\frac{(\bar{r}- r\Tilde{y}_1)
    (1-\Tilde{y}_1^2)^\frac{d-4}{2}}
    {\left(r^2+\bar{r}^2-2r\bar{r} \Tilde{y}_1 
    +(z-\bar{z})^2\right)^\frac{d}{2}}
    -\frac{(\bar{r}- r\Tilde{y}_1)
    (1-\Tilde{y}_1^2)^\frac{d-4}{2}}
    {\left(r^2+\bar{r}^2-2r\bar{r} \Tilde{y}_1 
    +(z+\bar{z})^2\right)^\frac{d}{2}}
    \right)
    \diff\Tilde{y}_1 \diff\bar{r} \diff\bar{z}.
\end{multline}
Again taking the substitution $\tau=-\Tilde{y}_1$, we find that
\begin{multline}
    \int_{-1}^0
    \left(\frac{(\bar{r}- r\Tilde{y}_1)
    (1-\Tilde{y}_1^2)^\frac{d-4}{2}}
    {\left(r^2+\bar{r}^2-2r\bar{r} \Tilde{y}_1 
    +(z-\bar{z})^2\right)^\frac{d}{2}}
    -\frac{(\bar{r}- r\Tilde{y}_1)
    (1-\Tilde{y}_1^2)^\frac{d-4}{2}}
    {\left(r^2+\bar{r}^2-2r\bar{r} \Tilde{y}_1 
    +(z+\bar{z})^2\right)^\frac{d}{2}}
    \right)\diff\Tilde{y}_1= \\
    \int_0^1
    \left(\frac{(\bar{r}+ r\tau)
    (1-\tau^2)^\frac{d-4}{2}}
    {\left(r^2+\bar{r}^2+2r\bar{r}\tau 
    +(z-\bar{z})^2\right)^\frac{d}{2}}
    -\frac{(\bar{r}+ r\tau)
    (1-\tau^2)^\frac{d-4}{2}}
    {\left(r^2+\bar{r}^2+2r\bar{r} \tau 
    +(z+\bar{z})^2\right)^\frac{d}{2}}
    \right)\diff\tau.
\end{multline}
Making the trivial substitution $\tau=\Tilde{y}_1$ for 
$0\leq\Tilde{y}_1\leq 1$,
this completes the proof of the identity \eqref{UZkernel}.

Finally, we observe that 
\begin{equation}
    H(r,-z,\bar{r},\bar{z})=H(r,z,\bar{r},\bar{z}),
\end{equation}
which implies that
\begin{equation}
    u_r(r,z)=u_r(r,-z).
\end{equation}
Likewise, we can observe that
\begin{equation}
    G(r,-z,\bar{r},\bar{z})=-G(r,z,\bar{r},\bar{z}),
\end{equation}
which implies that
\begin{equation}
    u_z(r,-z)=-u_z(r,z).
\end{equation}
This completes the proof.
\end{proof}

It is clear from the analysis in Section \ref{RegCritSection}, that blowup for solutions of the axisymmetric, swirl-free Euler equation in four and higher dimensions requires compression along the $z$-axis and stretching in the $r$-hyperplane. We will show that this is precisely what is afforded to us by vorticities which are odd in $z$, and nonnegative for $z>0$,
giving hyperbolic flow in the neighbourhood of a stagnation point at the origin.

\begin{proposition} \label{AxisVelocity}
Suppose $u\in H^s_{df}\left(\mathbb{R}^d\right), 
s>1+\frac{d}{2}, d\geq 3$, 
is axisymmetric and swirl-free. Further suppose that $\omega$ is odd in $z$, that for all $r,z>0, \omega(r,z)\geq 0$, and that $\omega$ is not identically zero.
Then for all $z>0$,
\begin{equation}
    u_z(0,z)<0,
\end{equation}
and for all $r>0$,
\begin{equation}
    u_r(r,0)>0.
\end{equation}
\end{proposition}

\begin{proof}
Taking the identity \eqref{UZkernel} from Proposition \ref{BiotSavartOdd}, we see that
\begin{multline}
    u_z(0,z)=-\frac{d-2}{\pi}
    \int_0^\infty \int_0^\infty 
    \bar{r}^{d-1} \omega(\bar{r},\bar{z})  
    \int_0^1
    (1-\tau^2)^\frac{d-4}{2}
    \\
    \left(\frac{1}
    {\left(\bar{r}^2 
    +(z-\bar{z})^2\right)^\frac{d}{2}}
    -\frac{1}
    {\left(\bar{r}^2
    +(z+\bar{z})^2\right)^\frac{d}{2}}\right)
    \diff\tau \diff\bar{r} \diff\bar{z}.
\end{multline}
Observe that for all $\bar{r},\bar{z},z>0$,
\begin{equation}
    \int_0^1
    (1-\tau^2)^\frac{d-4}{2}
    \\
    \left(\frac{1}
    {\left(\bar{r}^2 
    +(z-\bar{z})^2\right)^\frac{d}{2}}
    -\frac{1}
    {\left(\bar{r}^2
    +(z+\bar{z})^2\right)^\frac{d}{2}}\right)
    \diff\tau>0.
\end{equation}
This implies that for all $z>0$,
\begin{equation}
    u_z(0,z)<0.
\end{equation}

Likewise, applying the identity \eqref{URkernel} from Proposition \ref{BiotSavartOdd},
we can see that for all $r>0$,
\begin{multline}
    u_r(r,0)=
    \frac{d-2}{\pi}
    \int_0^\infty\int_0^\infty
    \bar{z}\bar{r}^{d-2} \omega(\bar{r},\bar{z})
    \int_0^1 
    \tau
    (1-\tau^2)^\frac{d-4}{2}\\ 
    \left(
    \frac{1}
    {\left(r^2+\bar{r}^2-2r\bar{r}\tau 
    +\bar{z}^2\right)^\frac{d}{2}} 
    -
    \frac{1}
    {\left(r^2+\bar{r}^2+2r\bar{r}\tau 
    +\bar{z}^2\right)^\frac{d}{2}}\right)
    \diff\tau \diff\bar{r} \diff\bar{z}.
\end{multline}
Observe that for all $\bar{r},\bar{z},r>0$,
\begin{equation}
    \int_0^1 
    \tau
    (1-\tau^2)^\frac{d-4}{2}\\ 
    \left(
    \frac{1}
    {\left(r^2+\bar{r}^2-2r\bar{r}\tau 
    +\bar{z}^2\right)^\frac{d}{2}} 
    -
    \frac{1}
    {\left(r^2+\bar{r}^2+2r\bar{r} \tau 
    +\bar{z}^2\right)^\frac{d}{2}} 
    \right)\diff\tau >0,
\end{equation}
and so for all $r>0$,
\begin{equation}
    u_r(r,0)>0.
\end{equation}
\end{proof}

\begin{proposition} \label{MirrorSymProp}
Suppose $u\in C\left([0,T_{max});
H^s_* \left(\mathbb{R}^d\right)\right)
\cap C^1\left([0,T_{max}),H^{s-1}_*
\left(\mathbb{R}^d\right)\right), 
d\geq 3, s>2+\frac{d}{2}$, is a solution of the Euler equation. Then the velocity $\Tilde{u} \in C\left([0,T_{max});
H^s_* \left(\mathbb{R}^d\right)\right)
\cap C^1\left([0,T_{max}),H^{s-1}_*
\left(\mathbb{R}^d\right)\right)$ is also a solution of the Euler equation, where
\begin{align}
    \Tilde{u}_r(r,z,t)&=u_r(r,-z,t) \\
    \Tilde{u}_z(r,z,t)&=-u_z(r,-z,t) \\
    \Tilde{\omega}(r,z,t)&=-\omega(r,-z,t).
\end{align}
\end{proposition}

\begin{proof}
First, we observe that the new velocity field is divergence free, with
    \begin{align}
        \nabla \cdot \tilde{u}(x)
        &=
        \partial_r\Tilde{u}_r(r,z,t)
        +\partial_z\Tilde{u}_z(r,z)
        +\frac{k}{r}\Tilde{u}_{r}(r,z,t) \\
        &=
        \left(\partial_r u_r+\partial_z u_z
        + \frac{k}{r}u_r\right)(r,-z,t) \\
        &=
        0.
    \end{align}
Likewise, taking the scalar curl, we confirm that
\begin{align}
    \partial_r \Tilde{u}_z(r,z)
    -\partial_z \Tilde{u}_z(r,z)
    &=
    -\partial_r u_z(r,-z)+\partial_z u_r(r,-z) \\
    &=
    -\omega(r,-z) \\
    &=\Tilde{\omega}(r,z),
\end{align}
and so we have confirmed $\Tilde{\omega}$ is the scalar vorticity associated with $\Tilde{u}$.
Finally, we compute that
\begin{align*}
    \left(\partial_t\tilde{\omega}
    +\Tilde{u}_r\partial_r\omega
    +\Tilde{u}_z\partial_z\omega
    +\frac{k}{r}\Tilde{u_r}\Tilde{\omega}
    \right)(r,z,t)
    &=
    \left(-\partial_t\tilde{\omega}
    -\Tilde{u}_r\partial_r\omega
    -\Tilde{u}_z\partial_z\omega
    -\frac{k}{r}\Tilde{u_r}\Tilde{\omega}
    \right)(r,-z,t) \\
    &=
    0.
\end{align*}
\end{proof}

\begin{theorem} \label{GeometryPreservedThm}
Suppose $u\in C\left([0,T_{max});
H^s_* \left(\mathbb{R}^d\right)\right)
\cap C^1\left([0,T_{max}),H^{s-1}_*
\left(\mathbb{R}^d\right)\right), 
d\geq 3, s>2+\frac{d}{2}$, is a solution of the Euler equation and that $\omega^0(r,z)$ is odd in $z$.
Then for all $0\leq t<T_{max},$ the vorticity $\omega(r,z,t)$ is odd in $z$ and 
\begin{equation}
    \int_0^\infty\int_0^\infty
    \omega(r,z,t)\diff r \diff z
    =
    \int_0^\infty\int_0^\infty
    \omega^0(r,z)\diff r \diff z
\end{equation}
Furthermore, if for all $r,z>0$
\begin{equation}
    \omega^0(r,z)\geq 0,
\end{equation}
then for all $0\leq t<T_{max}, r,z>0$
\begin{equation}
    \omega(r,z,t)\geq 0.
\end{equation}
\end{theorem}

\begin{proof}
We begin by defining $\Tilde{u}$ and $\Tilde{\omega}$ as in Proposition \ref{MirrorSymProp}. We know that $\Tilde{u}$ must also be a solution of the Euler equation and because $\omega^0(r,z)=-\omega^0(r,-z)$,
we can see that $\omega^0=\Tilde{\omega}^0$.
By uniqueness this implies that 
for all $0\leq t<T_{max}$,
\begin{equation}
    \omega(r,z,t)=\Tilde{\omega}(r,z,t)
    =-\omega(r,-z,t),
\end{equation}
and so we can conclude that $\omega$ being odd in $z$ is preserved by the dynamics. It then follows from Proposition \ref{BiotSavartOdd}, that $u_r(r,z)$ is even in $z$ and $u_z(r,z)$ is odd in $z$.

To prove the conservation of vorticity mass, we will consider the divergence form of the vorticity equation, $\partial_t\omega 
+\partial_r\left(\omega u_r\right)+
\partial_z \left(\omega u_z\right)=0$, computing that
\begin{multline}
    \frac{\diff}{\diff t}
    \int_0^\infty \int_0^\infty
    \omega(r,z,t) \diff r \diff z
    =
    -\int_0^\infty \int_0^\infty \partial_r(\omega(r,z,t)u_r(r,z,t))
    \diff r \diff z \\
    -\int_0^\infty \int_0^\infty 
    \partial_z(\omega(r,z,t)u_z(r,z,t))
    \diff z \diff r,
\end{multline}
and so
\begin{align*}
    \frac{\diff}{\diff t}
    \int_0^\infty \int_0^\infty
    \omega(r,z,t) \diff r \diff z
    &=
    \int_0^\infty \omega(0,z,t) u_r(0,z,t) \diff z
    +\int_0^\infty \omega(r,0,t)u_z(r,0,t) \diff r \\
    &=0,
\end{align*}
where we have used the fact that $\omega u_r$ and $\omega u_z$ decay at infinity, and that 
\begin{align}
    \omega(0,z,t)&=0 \\
    \omega(r,0,t) &=0,
\end{align}
due to regularity and oddness respectively.

Finally, we will observe that for all $r\in\mathbb{R}$,
\begin{equation}
    u_z(r,0)=0,
\end{equation}
because $u_z$ is odd in $z$. This implies stream lines cannot cross the plane $z=0$, and that therefore the set $\left\{(r,z):r,z>0\right\}$ is preserved by the flow map. We know that $\frac{\omega}{r^k}$ is transported by the velocity, so if for all $r,z>0$,
\begin{equation}
    \frac{\omega^0}{r^k}(r,z)\geq 0,
\end{equation}
then for all $0\leq t<T_{max}, r,z>0$,
\begin{equation}
    \frac{\omega}{r^k}(r,z,t)\geq 0.
\end{equation}
This completes the proof.
\end{proof}

We have now shown that in the geometric setting of odd scalar vorticities, non-negative when $z>0$, we can view $\omega(r,z,t)\diff r\diff z$ as a finite measure on $[0,+\infty]^2$ with constant total measure. Up to rescaling $\omega(r,z,t)\diff r\diff z$ is a probability measure. Therefore, it makes sense to consider the regularity problem in terms of the moments. We will show that a certain radial moment is increasing, and a certain vertical moment is decreasing.

\begin{theorem}
    Suppose $u\in C\left([0,T_{max});
H^s_* \left(\mathbb{R}^d\right)\right)
\cap C^1\left([0,T_{max}),H^{s-1}_*
\left(\mathbb{R}^d\right)\right), 
d\geq 3, s>2+\frac{d}{2}$, is a solution of the Euler equation and that $\omega(r,z,t)$ is odd in $z$, with for all $r,z>0$
\begin{equation}
    \omega(r,z,t)\geq 0.
\end{equation}
Then for all $0\leq t<T_{max}$,
\begin{equation*}
    \frac{\diff}{\diff t}
    \int_0^\infty \int_0^\infty
    z\omega(r,z,t) \diff r\diff z
    =
    -\frac{1}{2}\int_0^\infty
    u_z(0,z)^2 \diff z
    -k\int_0^\infty\int_0^\infty
    \frac{1}{r}u_r(r,z)^2\diff r\diff z.
\end{equation*}
\end{theorem}

\begin{proof}
    We will use the divergence form of the vorticity equation, 
    \begin{equation}
    \partial_t\omega+\partial_r(\omega u_r)
    +\partial_z(\omega u_z)=0.
    \end{equation}
    We can then compute that
    \begin{align}
    \frac{\diff}{\diff t}
    \int_0^\infty\int_0^\infty
    z\omega(r,z,t)\diff r \diff z
    &=
    -\int_0^\infty\int_0^\infty
    z \left(\partial_r(\omega u_r)
    +\partial_z(\omega u_z)\right) \\
    &=\int_0^\infty\int_0^\infty
    u_z \omega \diff r\diff z \\
    &= \int_0^\infty\int_0^\infty
    u_z (\partial_r u_z -\partial_z u_r) \diff r\diff z,
    \end{align}
    by integration by parts.
    Now calculate that
    \begin{align}
    \int_0^\infty\int_0^\infty
    u_z \partial_r u_z\diff r\diff z
    &=
    \int_0^\infty\int_0^\infty
    \frac{1}{2} \partial_r u_z^2 
    \diff r\diff z \\
    &=
    -\int_0^\infty u_z(0,z)^2\diff z.
    \end{align}
    Integrating by parts and applying the divergence free constraint, we find that
    \begin{align}
    -\int_0^\infty\int_0^\infty
    u_z  \partial_z u_r \diff r\diff z
    &=
    \int_0^\infty\int_0^\infty
    \partial_z u_z  u_r \diff r\diff z \\
    &=
    -k\int_0^\infty\int_0^\infty
      \frac{1}{r}u_r^2 \diff r\diff z 
      -\int_0^\infty\int_0^\infty
      u_r\partial_ru_r
      \diff r \diff z.
    \end{align}
    Finally observe that
    \begin{align}
    -\int_0^\infty\int_0^\infty
      u_r\partial_ru_r
      \diff r \diff z
      &=
      -\frac{1}{2}\int_0^\infty\int_0^\infty
      \partial_r u_r^2
      \diff r \diff z \\
      &=
      \int_0^\infty u_r(0,z)^2 \diff z \\
      &=0.
    \end{align}
\end{proof}

\begin{theorem}
    Suppose $u\in C\left([0,T_{max});
H^s_* \left(\mathbb{R}^d\right)\right)
\cap C^1\left([0,T_{max}),H^{s-1}_*
\left(\mathbb{R}^d\right)\right), 
d\geq 3, s>2+\frac{d}{2}$, is a solution of the Euler equation and that $\omega(r,z,t)$ is odd in $z$, with for all $r,z>0$
\begin{equation}
    \omega(r,z,t)\geq 0.
\end{equation}
Then for all $0\leq t<T_{max}$,
\begin{equation}
    \frac{\diff}{\diff t}
    \int_0^\infty \int_0^\infty
    r^{d-1}\omega(r,z,t) \diff r\diff z
    =
    \frac{d-1}{2}\int_0^\infty
    r^{d-2}u_r(r,0)^2 \diff r.
\end{equation}
\end{theorem}

\begin{proof}
    Following the same approach as above, and integrating by parts, we compute that
    \begin{align*}
    \frac{\diff}{\diff t}
    \int_0^\infty\int_0^\infty
    r^{d-1}\omega(r,z,t)\diff r \diff z
    &=
    -\int_0^\infty\int_0^\infty
    r^{d-1} \left(\partial_r(\omega u_r)
    +\partial_z(\omega u_z)\right) \\
    &=(d-1)\int_0^\infty\int_0^\infty
    r^{d-2 }u_r \omega \diff r\diff z \\
    &= (d-1)\int_0^\infty\int_0^\infty
    r^{d-2} u_r (\partial_r u_z -\partial_z u_r) \diff r\diff z.
    \end{align*}
    Next observe that
    \begin{equation}
    -\int_0^\infty\int_0^\infty
    r^{d-2}u_r\partial_z u_r \diff r \diff z
    =
    \frac{1}{2}\int_0^\infty 
    r^{d-2} u_r(r,0)^2\diff r.
    \end{equation}
    Using the form of the divergence free constraint
    \begin{equation}
    \partial_r \left(r^{d-2}u_r\right)
    +\partial_z \left(r^{d-2}u_z\right)
    =0,
    \end{equation}
    we can compute that
    \begin{align}
    \int_0^\infty\int_0^\infty
    r^{d-2} u_r \partial_r u_z \diff r \diff z
    &=
    -\int_0^\infty\int_0^\infty
    u_z \partial_r \left(r^{d-2}u_r\right)
    \diff r \diff z \\
    &=\int_0^\infty\int_0^\infty
    u_z \partial_z \left(r^{d-2}u_z\right) 
    \diff r \diff z \\
    &=
    -\frac{1}{2}\int_0^\infty 
    r^{d-2} u_z(r,0)^2\diff r \\
    &=
    0.
    \end{align}
    This completes the proof.
    \end{proof}

We will now consider a moment involving both $r$ and $z$ whose time derivative does not have a sign, but does relate in an interesting way to the partition of the energy between the radial and vertical directions.

\begin{proposition} \label{EnergyMomentProp}
    Suppose $u\in C\left([0,T_{max});
H^s_* \left(\mathbb{R}^d\right)\right)
\cap C^1\left([0,T_{max}),H^{s-1}_*
\left(\mathbb{R}^d\right)\right), 
d\geq 3, s>2+\frac{d}{2}$, is a solution of the Euler equation and that $\omega(r,z,t)$ is odd in $z$, with for all $r,z>0$
\begin{equation}
    \omega(r,z,t)\geq 0.
\end{equation}
Then for all $0\leq t<T_{max}$,
\begin{multline}
    \frac{\diff}{\diff t}
    \int_0^\infty \int_0^\infty
    r^{d-1} z\omega(r,z,t) \diff r\diff z
    =
    -(d-1)\int_0^\infty\int_0^\infty
    r^{d-2}u_z(r,z,t)^2\diff r \diff z \\
    +\int_0^\infty\int_0^\infty
    r^{d-2}u_r(r,z,t)^2\diff r \diff z.
\end{multline}
This can also be expressed as
\begin{equation}
    \frac{\diff}{\diff t}
    \int_0^\infty \int_0^\infty
    r^{d-1} z\omega(r,z,t) \diff r\diff z
    =
    -d K_z(t)
    +K_0,
\end{equation}
where
\begin{align}
    K_z(t)&= \int_0^\infty\int_0^\infty
    r^{d-2}u_z(r,z,t)^2\diff r \diff z \\
    K_0 &= \int_0^\infty\int_0^\infty
    r^{d-2}u^0_z(r,z)^2\diff r \diff z
    +\int_0^\infty\int_0^\infty
    r^{d-2}u^0_r(r,z)^2\diff r \diff z.
\end{align}
\end{proposition}

\begin{proof}
    Again starting with the divergence form of the vorticity equation and integrating by parts, we compute that
    \begin{align*}
    \frac{\diff}{\diff t}
    \int_0^\infty \int_0^\infty
    r^{d-1}z \omega(r,z,t) \diff r\diff z
    &=
    -\int_0^\infty \int_0^\infty
    r^{d-1}z \left(\partial_r(\omega u_r)
    +\partial_z (\omega u_z)\right)
    \diff r \diff z \\
    &=
    \int_0^\infty \int_0^\infty
    \left((d-1)r^{d-2}z \omega u_r 
    +r^{d-1} \omega u_z\right) \diff r \diff z.
    \end{align*}
    We will deal with these two terms separately.
    Again integrating by parts, we find that
    \begin{align*}
    \int_0^\infty \int_0^\infty
    r^{d-2}z \omega u_r  \diff r \diff z
    &=
    \int_0^\infty \int_0^\infty
    r^{d-2}z u_r (\partial_r u_z-\partial_z u_r)
    \diff r \diff z \\
    &=
    \int_0^\infty \int_0^\infty
    \left(-zu_z \partial_r\left(r^{d-2}u_r\right)
    -\frac{1}{2}r^{d-2}z \partial_z u_r^2 \right)
    \diff r \diff z\\
    &=
    \int_0^\infty \int_0^\infty
    \left(z u_z \partial_z\left(r^{d-2}u_z\right)
    +\frac{1}{2}r^{d-2}u_r^2 \right)
    \diff r \diff z \\
    &=
    \frac{1}{2}\int_0^\infty \int_0^\infty 
    \left(-r^{d-2} u_z^2
    +r^{d-2} u_r^2 \right) 
    \diff r \diff z.
    \end{align*}
    Next we compute that
    \begin{align*}
    \int_0^\infty\int_0^\infty 
    r^{d-1} \omega u_z \diff r \diff z
    &=
    \int_0^\infty\int_0^\infty 
    r^{d-1} (\partial_ru_z-\partial_z u_r) u_z
    \diff r \diff z \\
    &=
    -\frac{d-1}{2} \int_0^\infty\int_0^\infty 
    r^{d-2}u_z^2 \diff r \diff z
    +\int_0^\infty\int_0^\infty 
    r^{d-1} u_r \partial_z u_z \diff r \diff z.
    \end{align*}
    
    Putting these estimates together we find that
    \begin{multline} \label{EnergyMomentEqn}
    \frac{\diff}{\diff t}
    \int_0^\infty \int_0^\infty
    r^{d-1}z \omega(r,z,t) \diff r\diff z
    =-(d-1) \int_0^\infty\int_0^\infty 
    r^{d-2}u_z^2 \diff r \diff z \\
    +\frac{d-1}{2} \int_0^\infty\int_0^\infty 
    r^{d-2}u_r^2 \diff r \diff z
    +\int_0^\infty\int_0^\infty 
    r^{d-1} u_r \partial_z u_z \diff r \diff z.
    \end{multline}
    We will use the divergence free constraint to simplify the last term, computing that
    \begin{align}
    \int_0^\infty\int_0^\infty 
    r^{d-1} u_r \partial_z u_z \diff r \diff z
    &=
    \int_0^\infty\int_0^\infty 
    r^{d-1} u_r \left(-(d-2)\frac{u_r}{r}-\partial_r u_r\right) \diff r \diff z \\
    &\left(\frac{d-1}{2}-(d-2)\right)
    \int_0^\infty\int_0^\infty 
    r^{d-2}u_r^2 \diff r \diff z.
    \end{align}
    Plugging this back into \eqref{EnergyMomentEqn}, we find that
    \begin{multline}
    \frac{\diff}{\diff t}
    \int_0^\infty \int_0^\infty
    r^{d-1} z\omega(r,z,t) \diff r\diff z
    =
    -(d-1)\int_0^\infty\int_0^\infty
    r^{d-2}u_z(r,z,t)^2\diff r \diff z \\
    +\int_0^\infty\int_0^\infty
    r^{d-2}u_r(r,z,t)^2\diff r \diff z.
\end{multline}
The result then follows by conservation of energy by making the substitution
\begin{equation}
    K_r(t)=K_0-K_z(t).
\end{equation}
\end{proof}

We will now use Proposition \ref{EnergyMomentProp} to prove Theorem \ref{ConditionalBlowupThmIntro}, our conditional blowup result for the Euler equation with the colliding vortex tube geometry, which is restated for the reader's convenience.

\begin{theorem} \label{ConditionalBlowupThm}
    Suppose $u\in C\left([0,T_{max});
H^s_* \left(\mathbb{R}^d\right)\right)
\cap C^1\left([0,T_{max}),H^{s-1}_*
\left(\mathbb{R}^d\right)\right), 
d\geq 4, s>2+\frac{d}{2}$, is a solution of the Euler equation and that $\omega^0(r,z)$ is odd in $z$, with for all $r,z>0$
\begin{equation}
    \omega^0(r,z)\geq 0,
\end{equation}
and that $\omega^0$ is not identically zero.
Further suppose that there exists $\epsilon>0$ such that for all $0\leq t<T_{max}$,
\begin{equation}
    K_z(t)\geq \frac{1+\epsilon}{d}K_0.
\end{equation}
Then this solution of the Euler equation blows up in finite-time with
\begin{equation}
    T_{max}\leq \frac{1}{\epsilon K_0} \int_0^\infty \int_0^\infty
    r^{d-1}z \omega^0(r,z) \diff r\diff z.
\end{equation}
\end{theorem}

\begin{proof}
    Applying Theorem \ref{GeometryPreservedThm}, we can see that the oddness and positivity conditions on the vorticity are both preserved by the dynamics, and applying Proposition \ref{EnergyMomentProp}, we find that
    for all $0\leq t<T_{max}$
    \begin{align}
    \frac{\diff}{\diff t}
    \int_0^\infty \int_0^\infty
    r^{d-1} z\omega(r,z,t) \diff r\diff z
    &=
    -d K_z(t)
    +K_0 \\
    &\leq -\epsilon K_0.
    \end{align}
    Integrating this differential inequality we find that
    for all $0\leq t<T_{max}$
    \begin{equation}
    \int_0^\infty \int_0^\infty
    r^{d-1} z\omega(r,z,t) \diff r\diff z
    \leq 
    \int_0^\infty \int_0^\infty
    r^{d-1} z\omega^0(r,z) \diff r\diff z
    -\epsilon K_0 t.
    \end{equation}
    Due to non-negativity, we can see that for all $0\leq t<T_{max}$,
    \begin{equation}
    \int_0^\infty \int_0^\infty
    r^{d-1} z\omega(r,z,t) \diff r\diff z
    >0.
    \end{equation}
    Taking the limit $t\to T_{max}$, we can see that
    \begin{equation}
    \epsilon K_0 T_{max}
    \leq 
    \int_0^\infty \int_0^\infty
    r^{d-1} z\omega^0(r,z) \diff r\diff z,
    \end{equation}
    and this completes the proof.
\end{proof}

\begin{remark} \label{BlowupConjecture}
The most natural blowup scenario to consider for axisymmetric, swirl-free solutions in four and higher dimensions are vorticities that are odd in $z$, positive on the upper half plane, and vanish only linearly at the axis of symmetry, leading to $\frac{\omega^0}{r^{d-2}}\notin L^\infty$. 
Theorem \ref{ConditionalBlowupThm} suggests the possibility that there exists some $d_0\geq 4$, such that for all $d\geq d_0$, there exists $u^0\in H^s_*\left(\mathbb{R}^d\right), s>2+\frac{d}{2}$, with $\omega^0$ odd in $z$, for all $r,z>0, \omega^0(r,z)>0$, and 
$\frac{\omega^0}{r^{d-2}}\notin L^\infty$, such that the solution of the Euler equation with this initial data blows up in finite-time.

This is precisely the geometric setup given in Elgindi's recent blowup result for axisymmetric, swirl-free $C^{1,\alpha}$ solutions of the Euler equation in three dimensions \cite{Elgindi}. Similar to the geometric setup proposed for finite-time blowup here, the vorticity Elgindi's blowup solutions are odd in $z$ and positive in the upper half plane, and involve blowup driven by compression along the $z$ axis, and stretching in the horizontal plane.
Furthermore, a key aspect of Elgindi's blowup result is that $\frac{\omega^0}{r}\notin L^\infty$.
This suggests that a similar approach may be able to yield smooth solutions of the Euler equation in four and higher dimensions that blowup in finite-time when $\frac{\omega^0}{r^k}\notin L^\infty$.
This is also the geometry used by Choi and Jeong to prove growth of the vorticity at infinity of the form
\begin{equation}
    \|\omega(\cdot,t)\|_{L^\infty}
    >C(1+t)^{\frac{1}{15}-\epsilon},
\end{equation}
for smooth, axisymmetric, swirl-free solutions of the Euler equation in three dimensions \cite{ChoiJeong}.
This class of data was also investigated by Childress in three dimensions as a potential case of vorticity growth \cite{Childress}.
\end{remark}

\begin{remark}
Even in two dimensions, the geometric scenario considered in Remark \ref{BlowupConjecture}, has a substantial precedent in the literature.
There are a great number of papers in the literature of the two dimensional Euler equation based on vorticities $\omega$ that are 
odd with respect to both $x_1$ and $x_2$,
\begin{align}
    \omega(-x_1,x_2)
    &=-\omega(x_1,x_2) \\
    \omega(x_1,-x_2)
    &=
    -\omega(x_1,x_2),
\end{align}
and nonnegative in the first quadrant, with $\omega(x_1,x_2)\geq 0$ for $x_1,x_2>0$.
For examples, Iftimie, Gamblin, and Sideris used a vorticity with this structure to prove linear growth in the radius of support in two dimensions \cite{Iftimie}, showing that
\begin{equation}
    \frac{\diff R}{\diff t}
    \geq \kappa\left(\omega^0\right),
\end{equation}
where $\kappa\left(\omega^0\right)>0$ is a constant depending on the initial vorticity.
This is also the geometry of the log-Lipshitz stationary solution of the Euler equation considered by Bahouri and Chemin \cite{BahouriChemin} that exhibits double exponential growth in the Lagrangian map. Kiselev and \v{S}ver\'{a}k used a very similar geometry \cite{KiselevSverak}, but with a boundary, to prove double exponential growth of the vorticity gradient for solutions of the Euler equation in two dimensions.

We should note that this is precisely the two dimensional analogue geometric setup proposed for finite-time blowup in four and higher dimensions in Remark \ref{BlowupConjecture}.
Note that in two dimensions
$r=|x_1|$ and $e_r=\sgn(x_1)e_1$,
and so $u$ is axisymmetric if and only if $u_1$ is odd in $x_1$ and $u_2$ is even in $x_1$.
Recall that $\omega=\partial_1 u_2-\partial_2 u_1$,
and we can see that if $u$ is axisymmetric implies that $\omega$ is odd with respect to $x_1$. Using the Biot-Savart law, it can be seen that in fact $u$ is axisymmetric in two dimensions if and only if $\omega$ is odd with respect to $x_1$.
Furthermore, in two dimensions any axisymmetric vector field is automatically swirl-free, as $\mathbb{R}^2=\spn\left\{e_r,e_z\right\}$.
Therefore, the requirement that $\omega$ is odd with respect to $x_1$ in two dimensions is analogous to the requirement that $u$ is axisymmetric, swirl-free in three and higher dimensions. Likewise the requirement that $\omega$ is odd with respect to $x_2$ in two dimensions is analogous to the requirement that $\omega$ is odd with respect to $z$ in three and higher dimensions. Finally the requirement that $\omega$ is nonnegative in the first quadrant in two dimensions is analogous to the requirement that $\omega$ is nonnegative for $z>0$ in three and higher dimensions. This shows that the geometric setup proposed for finite-time blowup in four and higher dimensions in Remark \ref{BlowupConjecture} is entirely analogous to the geometric setup to previous works giving lower bounds based on a hyperbolic stagnation point at the origin.
\end{remark}

\appendix

\section{A stream function approach to the Biot-Savart Law} \label{StreamFunctionSection}

In this appendix, we will give a second derivation of the Biot-Savart Law 
(Proposition \ref{BiotSavartCoorindate1}) 
for axisymmetric no-swirl divergence-free vector fields in $\R^d$, $d\ge3$, using the stream function.
For the 3 dimensional axisymmetric flows, the stream function is given in Feng-Sverak \cite{FengSverak}*{Page 97}, 
\EQ{
\psi(r,z) = \frac 1{2\pi} \iint_\Pi \int_0^\pi \frac{r \bar r \cos \th}{ \left(r^2+\rb^2-2r\rb \cos \th 
    +(z-\zb)^2\right)^{1/2}} d\th \,\om_\th(\bar r,\bar z) d\bar rd\bar z
}
which is also formula (2.3) in \cite{ChoiJeong}. Here $\Pi = (0,\infty)\times \R$ for $(\bar r,\bar z)$, and
$\om_\th = \pd_z u_r- \pd_r u_z$ has an opposite sign as $\om$. Note that another common choice of the  stream function for $d=3$ is $ \psi_\th = r^{-1} \psi$  with
$-\De( \psi_\th e_\th) = \om_\th e_\th$, see \cite{nslec}*{(10.12)}.

For general dimension $d=k+2 \ge 3$,
due to $\div u=0$, 
\[
\pd_r (r^k   u_r ) + \pd_z (r^k   u_z)=0,
\]
there is $\psi(r,z)$ 
such that
\EQ{\label{ur-uz-psi}
r^k   u_r =- \pd_z \psi, \quad
r^k   u_z =  \pd_r \psi.
}
Thus
\EQ{
\om=\pd_r u_z - \pd_z u_r
= \pd_r (r^{-k}   \pd_r \psi)-\pd_z(-r^{-k}   \pd_z \psi)
= r^{-k}   \bke{ \pd_r^2  +\pd_z^2 -\frac kr \pd_r } \psi.
}
Let
$
\psi = r^{k+1} \eta$.
Then
\EQ{
\om=  r \bke{ \pd_r^2 + \frac{k+2}r \pd_r 
+  \pd_z^2  } \eta
= r \De_{\R^{k+3} \times \R} \eta.
}
Hence we have $\eta=(\De_{\R^{k+3} \times \R})^{-1}\Om$, $\Om=\frac{\om}{r}$. We can compute
$\eta(r,z)$ at $re_1+ze_z$ where $e_z$ is the last standard basis vector in $\R^{k+3} \times \R$,
\EQ{
\eta(r,z) = \int_{\R^{k+3} \times \R} \frac {-C_0}{|re_1+ z e_{z} - x|^{k+2}} \,\Om(x) dx,
}
where $C_0=\frac{\Ga(\frac{d}2)} {4\pi^{\frac d2+1}}$ is the constant of the Newton kernel in $\R^{d+2}$.
Write $x=(\bar r y, \bar z) \in \R^{k+3} \times \R$ with $(\bar r, \bar z) \in \Pi$ and $y=(y_1,y') \in \mathbb{S}^{d}$, $y_1 \in [-1,1]$ being the first component of $y$. We have
\[
|re_1+ z e_{z} - x|^2=(r-\bar r y_1)^2 + |\bar r y'|^2+ (z-\bar z)^2=(r-\bar r)^2 +  2r\bar r(1-y_1) + (z -\bar z)^2,
\]
and
\EQ{
\eta(r,z) = \iint_\Pi \int_{\mathbb{S}^{d}} \frac {-C_0} {
\bkt{(r-\bar r)^2 +  2r\bar r(1-y_1) + (z -\bar z)^2}^{\frac{d}2}} dS(y) \,\frac{\om(\bar r,\bar z)}{\bar r} \,d\bar r\,d\bar z.
}

By Lemma \ref{SphericalIntegration},
\EQ{
\eta(r,z)
= -C_1\iint_\Pi \int_0^\pi \frac {(\sin \th)^{k+1} \,d\th}
{X^{\frac{d}2}}\, \,\om(\bar r,\bar z)\bar r^{k+1}  \,d\bar r \,d\bar z,
}
where 
$C_1=C_0m_{k+1} = \frac{\Ga(\frac{d}2)} {4\pi^{\frac d2+1}}\cdot  \frac{2 \pi^{\frac{d}2}}{\Ga(\frac{d}2)} =\frac1{2\pi}$,
and
\EQ{
X=(r-\bar r)^2 +  2r\bar r(1-\cos \th) + (z -\bar z)^2 .
}
Thus
\EQ{\label{psi.int1}
\psi(r,z)=-\frac 1{2\pi}\iint_\Pi \int_0^\pi \frac {(\sin \th)^{k+1} }
{X^{\frac{d}2}}\,d\th \,\om(\bar r,\bar z)(r\bar r)^{k+1}  \,d\bar r \,d\bar z.
}

From \eqref{psi.int1} and \eqref{ur-uz-psi} that $u_r(r,z) = - \frac 1{r^k}\pd_z \psi$,
\EQ{\label{ur-cFk}
u_r(r,z)  = -
\frac d{2\pi}\iint_\Pi \int_0^\pi \frac {(\sin \th)^{k+1}}
{X^{\frac{d}2+1}} \,d\th \,(z-\bar z)\om(\bar r,\bar z)r\bar r^{k+1}  \,d\bar r \,d\bar z.
}
We can rewrite it using $\pd_\th X=2r\bar r \sin\th$ and integrate by parts in $\th$,
\EQS{\label{ur-cFk2}
u_r(r,z)  & = 
\frac 1{2\pi}\iint_\Pi \int_0^\pi(\sin \th)^{k} \pd_\th \bke{ \frac {1}
{X^{\frac{d}2}}}\,d\th \,(z-\bar z)\om(\bar r,\bar z)\bar r^{k}  \,d\bar r \,d\bar z
\\
&=-\frac k{2\pi}\iint_\Pi \int_0^\pi\frac{(\sin \th)^{k-1}\cos \th} 
{X^{\frac{d}2}}\,d\th \,(z-\bar z)\om(\bar r,\bar z)\bar r^{k}  \,d\bar r \,d\bar z.
}
This recovers \eqref{radvelo2}.

We can also rewrite $\psi$ in \eqref{psi.int1} as in \eqref{ur-cFk2} to get
\EQS{ \label{psi.int2}
\psi(r,z)& =\frac 1{2\pi}\iint_\Pi \int_0^\pi  {(\sin \th)^{k} } \pd_\th \bke{ \frac {1}
{kX^{\frac{k}2}}}\,d\th
\, \,\om(\bar r,\bar z)(r\bar r)^{k}  \,d\bar r \,d\bar z
\\
& =-\frac 1{2\pi}\iint_\Pi \int_0^\pi \frac {(\sin \th)^{k-1}\cos \th } 
{X^{\frac{k}2}}\,d\th
\, \,\om(\bar r,\bar z)(r\bar r)^{k}  \,d\bar r \,d\bar z.
}
From \eqref{psi.int2} and \eqref{ur-uz-psi} that $u_z(r,z) = \frac 1{r^k}  \pd_r\psi$, 
 and using $\pd_r X=2r - 2\bar r \cos \th$,
\begin{align*}
    u_z(r,z) 
    &=
    \frac 1{2\pi}\iint_\Pi \int_0^\pi (\sin \th)^{k-1}\cos \th\bke{ \frac { k(r - \bar r \cos \th)}
{X^{\frac{k}2+1} } - \frac { k}{rX^{\frac{k}2} }}  d\th \,\om(\bar r,\bar z)\bar r^{k}  \,d\bar r \,d\bar z \\
&=I_1+I_2.
\end{align*}

We rewrite $I_2$ and integrate by parts to get
\EQS{
I_2&=-\frac 1{2\pi} \iint_\Pi \int_0^\pi  \frac {1}
{rX^{\frac{k}2}} \pd_\th \bke{ (\sin \th)^{k}}  d\th \,\om(\bar r,\bar z)\bar r^{k}  \,d\bar r \,d\bar z
\\
&=-\frac {k}{2\pi} \iint_\Pi \int_0^\pi    \frac {(\sin \th)^{k+1} }
{X^{\frac{d}2}}  \,d\th \,\om(\bar r,\bar z)\bar r^{k+1}  \,d\bar r \,d\bar z.
}
Thus
\begin{align*}
    u_z(r,z)& =
I_1+I_2
\\
&=\frac k{2\pi} \iint_\Pi \int_0^\pi \frac{(\sin \th)^{k-1}}{X^{\frac{d}2}} 
\bigg[ (\cos \th) (r - \bar r \cos \th)  -(\sin \th)^2\rb  \bigg]  \,d\th \,\om(\bar r,\bar z)\bar r^{k}  \,d\bar r \,d\bar z
\\
&=  \frac k{2\pi} \iint_\Pi \int_0^\pi \frac{(\sin \th)^{k-1}}{X^{\frac{d}2}} 
( r\cos \th-\rb  )  \,d\th \,\om(\bar r,\bar z)\bar r^{k}  \,d\bar r \,d\bar z.
\end{align*}
This recovers \eqref{horvelo2}.

We can rewrite the steam function formulas in \eqref{psi.int1} and \eqref{psi.int2} as
\begin{equation}
  \psi(r,z,t) = -\frac{1}{2\pi} \iint\limits_{\Pi} \cF(S)  \; (r \rb)^{\frac{d}{2}-1} 
  \;\om(\rb,\zb,t) d\rb d\zb \;,
\end{equation}
where 
\[
  S = \frac{(r-\rb)^2 + (z-\zb)^2}{r \rb},\quad X= r\rb (2-2\cos\th + S),
\]
and for $s > 0$,
\EQ{
\cF(s) = \int_0^\pi  \frac{(\sin \th)^{k+1}}{[2(1-\cos \th) + s]^{\frac d2}}\,d\th
 = \int_0^\pi  \frac{(\sin \th)^{k-1} \cos \th}{[2(1-\cos \th) + s]^{\frac k2}}\,d\th.
}
This form isolates the dependence on length variables $r,r',z,z'$ in
the angular integral.
The velocity components are recovered from the stream function 
via~\eqref{ur-uz-psi}, 
resulting in the expressions
\EQS{
u_r(r,z,t) =  \frac {1} {\pi r^{d-2}} 
\iint\limits_{\Pi} 
\cF'(S) (z-\zb) (r \rb)^{\frac{d}{2}-2} \; \om(\rb,\zb,t) d\rb d\zb }
and
\EQS{
u_z(r,z,t) = -\frac {1} {2\pi r^{d-2}} 
\iint\limits_{\Pi} \bke{\cF'(S) \p_r S +
\frac{d-2}{2r} \cF(S)}
(r\bar r )^{\frac{d}{2}-1} \; \om(\rb,\zb,t) d\rb d\zb .}

\section*{Acknowledgements}
The authors would like to thank Stephen Gustafson for his helpful advice on this problem. The authors would also like to thank the referees for their thorough reading of the manuscript and helpful suggestions.

Both authors were partially supported by NSERC under the grants RGPIN-2023-04534 and  RGPIN-2018-04137. EM was also supported by the Pacific Institute for the Mathematical Sciences as a PIMS postdoctoral fellow at the University of British Columbia. This PIMS postdoctoral fellowship was supported by NSERC under grant no. 568577-2022.

\bibliography{bib}

\end{document}